\documentclass[11pt,twoside]{article}
\usepackage{graphicx} 
\usepackage{times}
\usepackage{amsmath,amssymb,amsthm}
\usepackage{enumerate}
\usepackage{cite}
\usepackage{epstopdf}
\usepackage{color}

\usepackage[top=2.5cm, bottom=2.5cm, left=2.5cm, right=2.5cm]{geometry}

\allowdisplaybreaks

\newcommand{\R}{\mathbb R}
\newcommand{\Z}{\mathbb Z}
\newcommand{\p}{\partial}
\newcommand{\ve}{\varepsilon}
\newcommand{\f}{\frac}
\newcommand{\al}{\alpha}

\newcommand{\ds}{\displaystyle}
\newcommand{\lc}{\left(}
\newcommand{\rc}{\right)}
\newcommand{\lp}{\left\|}
\newcommand{\rp}{\right\|}
\newcommand{\md}{\mathrm{d}}
\newcommand{\ra}{[2,\infty)\times\mathbb{R}^n}
\newcommand{\rb}{[2,t)\times\mathbb{R}^n}
\theoremstyle{plain}
\newtheorem{theorem}{Theorem}[section]
\newtheorem{proposition}{Proposition}[section]
\newtheorem{lemma}[theorem]{Lemma}

\newtheorem{remark}[theorem]{Remark}
\newcommand{\crit}{\textup{crit}}
\newcommand{\conf}{\textup{conf}}

\numberwithin{equation}{section}

\title{Morawetz type estimate for damped wave equation in $\R^n (n\ge 4)$ and its application}
\author{Daoyin He$^{1}$, \qquad Ning-An Lai $^{2}$\vspace{0.5cm}\\
	\small 1.
	School of Mathematics, Southeast University, Nanjing
	210089, China, 101012711@seu.edu.cn.\\
	\small 2. School of Mathematical Sciences, Zhejiang Normal University,
	Jinhua 321004, China, ninganlai@zjnu.edu.cn.\\}
\vspace{0.5cm}

\begin{document}
	
	\maketitle
	
	\begin{abstract}
		
		In this paper we establish a Morawetz type etimate for the linear inhomogeneous wave equation with time-dependent scale invariant damping in $\R^n (n\ge 4)$. The novelty is that we view the differential operator $\Box+\frac{\mu}{t}\partial_t$ as $n+1+\mu$ dimensional operator, then a well-matched multiplier is introduced. As an application, a sharp global existence result for the small data Cauchy problem of the semilinear wave equation
		\[
		\p_t^2u-\Delta u+\frac{\p_tu}{t}=|u|^p,~~~t>t_0\ge 0
		\]
		is obtained in $\R^n (n\ge 4)$.
	\end{abstract}
	
	
	\vskip 0.1 true cm
	
	\section{Introduction}
	
	In this article, we study the Cauchy problem for the semilinear wave equation with a scale-invariant damping term
	\begin{equation}\label{equ:original}
		\partial_t^2 \phi-\Delta \phi +\f{\mu}{t}\,\p_t\phi=|\phi|^p,
		\quad t>t_0, x\in\R^n
	\end{equation}
	where \(n\geq4\), \(\mu>0\), \(p>1\), and either \(t_0=0\) (singular problem) or \(t_0>0\) (regular problem). This kind of problem admits both interesting physical background and mathematical structure. It is closely related to the compressible Euler equations with time
	dependent damping, which can be used to describe the
	movement of smooth supersonic polytropic gases in a long divergent De Laval nozzle, see the introduction in \cite{HLYin2D} and references therein. Mathematically, the behavior of the solution can be both wave-like or heat-like, depending on the size of the positive constant $\mu$. Generally speaking, if $\mu$ is small, then the wave type equation dominates the heat type one, while if $\mu$ is big enough, the heat type equation will dominate.
	Hence we can expect a threshold of the constant $\mu_c(n)$, see the introduction in \cite{IKTW20, LST20} and references therein.

	There is extensive literature on the long time behavior (finite time blow-up or global existence) for these two kind problems, especially for the regular one. Next we provide a brief but far from exhaustive overview of the existing results. We start with the regular Cauchy problem
	\begin{equation}
		\label{REPD}
		\left\{
		\begin{aligned}
			& u_{tt}-\Delta u+\frac{\mu}{t}u_t=|u|^p, ~~~~t\ge t_0>0, x\in \R^n,\\
			& u(t_0, x)=\ve u_0(x), \quad u_t(t_0, x)=\ve u_1(x),~~~~x\in \R^n
		\end{aligned}
		\right.
	\end{equation}
	or its analogy form
	\begin{equation}
		\label{Sdamped}
		\left\{
		\begin{aligned}
			& u_{tt}-\Delta u+\frac{\mu}{1+t}u_t=|u|^p, ~~~~t\ge 0, x\in \R^n,\\
			& u(0, x)=\ve u_0(x), \quad u_t(0, x)=\ve u_1(x),~~~~x\in \R^n.
		\end{aligned}
		\right.
	\end{equation}
	People conjecture that there exists a threshold constant 
	\begin{equation*}
		\mu_c(n)=\frac{n^2+n+2}{n+2},
	\end{equation*}
	such that the small data Cauchy problem \eqref{REPD} or \eqref{Sdamped} admits a critical power 
	\begin{equation*}
		p_{\crit}=p_{\crit}(n,\mu)=\left\{
		\begin{aligned}
			& p_S(n+\mu), &&\text{if}\ \ 0\le \mu<\mu_c(n),\\
			& p_F(n), &&\text{if}\ \ \mu\ge\mu_c(n),\\
		\end{aligned}
		\right.
	\end{equation*}
	where ``critical" means that if $1<p<p_c$, then the corresponding problem has no global solution while if $p>p_c(n)$ the solution exists globally in time. And $p_S(n)$ denotes the positive root for the quadratic equation
	\begin{equation}\label{equ:strauss}
		(n-1)p^2-(n+1)p-2=0,
	\end{equation}
	and the critical exponent for the small data Cauchy problem of the classical semilinear wave equation
	\[
	u_{tt}-\Delta u=|u|^p,
	\]
	while $p_F(n)$ denotes the critical exponent for the small data Cauchy problem of the semilinear heat equation
	\[
	u_t-\Delta u=|u|^p,
	\]
	see \cite{Fuj, Gls1, Joh, LZ0, Sid, Strauss, Yor, Zho07} and references therein. The conjecture of the existence of $\mu_c(n)$ has appeared in many works, we refer to \cite{IS, IKTW20, LST20, HLYin2D} for example. This conjecture has been verified for a special case, i.e., $\mu=2$, at least in the low dimensions $(\R^n, 1\le n\le 3)$, see \cite{DL, Rei1, P, W1, IKTW20, Kato, Lai20, Wak16}. This special case is essentially the classical semilinear wave equation without damping, since after taking the so called Liouville transform
	\[
	v(t, x)=(1+t)^{\frac{\mu}{2}}u(t, x)
	\]
	in the equations $\eqref{Sdamped}_1$, we may get the equation for $v(t, x)$
	\begin{equation*}
		\begin{aligned}
			v_{tt}-\Delta v+\frac{\mu(2-\mu)}{4(1+t)^2}v=\frac{|v|^p}{(1+t)^{\frac{\mu(p-1)}{2}}},
		\end{aligned}
	\end{equation*}
	where the third mass term in the left side will disappear if $\mu=2$.
	
	For the general constant $\mu\neq 2$, the finite time blow up part of the conjecture mentioned above has been proved, see \cite{W1, IS, Pal1, LPT}. Recently, the global existence attracts much attention. By using a Morawetz type estimate, the second author \cite{LZ} established global existence result of radial solution in $\R^3$ for $\mu\in \left[\frac32, 2\right)$. In Li-Wang-Yin \cite{LWY25} and He-Li-Yin \cite{HLYin2D}, the global existence of weak solution in $\R^2$ for
	\[
	0<\mu <2, \mu\neq 1, p_S(2+\mu)<p<p_{\conf}(2, \mu)
	\]
	and 
	\[
	0<\mu <2, \mu\neq 1, p\ge p_{\conf}(2, \mu)
	\]
	was proved respectively, where the conformal exponent $p_{\conf}$ is defined in \eqref{pconf} below. What is more, it is announced in \cite{HLYin2D} that the global existence result for $\mu=1$ and $\mu>2$ will be proved in a forthcoming paper, hence by combing the corresponding result for $\mu=2$ in 
	\cite{DA, Rei1}, we may claim that the conjecture in $\R^2$ has been verified completely, while for the $1-D$ case we refer to \cite{DA21}. He-Sun-Zhang \cite{HSZ} studied the high dimensional case $(\R^n, n\ge 3)$, and established global existence result for 
	\[
	\mu\in (0, 1)\cup (1, 2), p_S(n, \mu)<p\le p_{\conf}(n, \mu).
	\]
	The method of the analysis in \cite{HLYin2D, HSZ} is based on making a change of time variable and transforming \eqref{Sdamped} into the corresponding Tricomi equation. However, this kind of transform does not work for \(\mu=1\), thus the case of \(\mu=1\) was not solved in \cite{HLYin2D, HSZ}.
	
	If $t_0=0$ in \eqref{REPD}, then it becomes to the singular problem, which is much more hard than the corresponding regular problem. The known results and references are relatively few, comparing to the regular one. D'Abbicco \cite{DA21} studied the global existence of weak solution to both the singular problem \eqref{REPD} and the regular problem \eqref{Sdamped} in $1-D$, and proved that there exists a unique global-in-time weak solution in $L^\infty\left([t_0, \infty), L^p\right)$ for $p>\max\left\{p_S(1+\mu), 3\right\}$. Fan-Lai-Takamura \cite{FLT25} and \cite{FLT252} established finite time blow up result for $1<p<p_S(n+\mu)$ in $\R^n, n\ge 1$ and $p=p_S(n+\mu)$ in $\R^n, n\ge 2$ respectively. However, due to the singularity, the conjecture for the singular Cauchy problem is far from being solved completely.
	
	We should mention that the wave equations with scale invariant damping are closely related to the Tricomi equations, which have significant application in fluid dynamics, we refer the reader to \cite{HWY1, HWY2, HWY3, HWY4, HWY5, LiG25} and references therein for more details.
	
	In order to prove global existence of solution to the small data Cauchy problem \eqref{REPD}, the key estimates are weighted Strichartz inequalities for linear inhomogeneous equation, which can be obtained by interpolation from \(L^1-L^\infty\) estimate and \(L^2-L^2\) estimate. However for the latter case, the establishment of weighted \(L^2-L^2\) estimate is difficult and requires techniques from degenerated Fourier integral operators. In order to avoid such complicated and involved computations, motivated by \cite{LZ0}, we turn to establish the weighted \(L^2-L^2\) estimate by Morawetz type energy estimate, which is obtained by introducing an appropriate multiplier (see \eqref{m1} below) and then integrating on the light cone. Comparing the multiplier \eqref{m1} with the one in \cite{LZ0} for the classical wave equations, one additional term $\frac{\mu}{2t}$ appears, which aims to be adapted to the scale invariant damping term $\frac{\mu}{t}u_t$ in the equation. Actually, the main part of the multiplier \eqref{m1} corresponds to the $n$ dimensional radial Laplace operator $\partial_r^2+\frac{n-1}{r}\partial_r$ is $\partial_r+\frac{n-1}{2r}$,
	which motivates us to see the operator $\partial_t^2+\frac{\mu}{t}\partial_t$ as $\mu+1$ dimensional radial Laplace operator with respect to the time variable $t$, and hence it is natural to add the term $\frac{\mu}{2t}$ in the multiplier.

	We come to our key weighted $L^2$-$L^2$ estimate. Consider the following Cauchy problem for the linear inhomogeneous wave equation
	\begin{equation}
		\begin{cases}
			&\partial_t^2 \phi-\triangle \phi+\frac{\mu}{t}\partial_t\phi=F(t,x), \quad (t,x)\in(t_0,\infty)\times\R^n, \\
			&\phi(t_0,x)=0,\quad \partial_t\phi(t_0,x)=0.
		\end{cases}
		\label{equ:linear-inhomo}
	\end{equation}
	Denote
	\begin{equation*}
		u=t-|x|, \quad \underline{u}=t+|x|, \quad r=|x|,
	\end{equation*}
	then our main result states as follows:
	\begin{theorem}\label{thm:Morawetz}
		Consider the Cauchy problem for the linear wave equation \eqref{equ:linear-inhomo} with \(n\geq4\) and \(t_0\geq0\).
		Then for \(\mu\in(0,2)\) and \(\gamma>1\), it holds
		\begin{equation}\label{equ:Morawetz}
			\begin{split}
				&\sup_{t_0< t\leq T}t^{\frac{\mu}{2}}\Bigg(\lp(1+|u|)^{\frac{1}{2}}
				\nabla_{t,x}\phi(t,\cdot)\rp_{L^2(\R^n)}
				+\lp(1+|u|)^{\frac{1}{2}}\frac{\phi(t,\cdot)}{r}\rp_{L^2(\R^n)}+\lp(1+|u|)^{\frac{1}{2}}\frac{\phi(t,\cdot)}{t}\rp_{L^2(\R^n)}
				\Bigg) \\
				\leq &C\lp(1+|u|)^{\frac{\gamma}{2}}
				(1+|\underline{u}|)^{\frac{\gamma}{2}}
				t^{\frac{\mu}{2}}F(t,x)\rp_{L^2([t_0,T]\times\R^n)},
			\end{split}
		\end{equation}
		where the positive constant \(C\) depends on \(n\), \(\mu\) and \(\gamma\).
	\end{theorem}
	As an application of our result, we consider global existence of small amplitude solution for \eqref{equ:original}. Recall that the critical exponent \(p_{\crit}(n,\mu)\) of \eqref{equ:original} is the smallest exponent \(p_{\crit}\) such that, if \(p>p_{\crit}\), then there exists a unique global small data solution to the problem, whereas if \(1<p\leq p_{\crit}\), then the solution blows up in finite time.
	For the case of \(\mu=0\), or the classical wave equation, the critical exponent \(p_{\crit}\) is the positive root of the quadratic equation \eqref{equ:strauss}.

	It is well known that when the damping term \(\frac{\mu}{t}\partial_tu\) appears, we have a shift from \(n\) to \(n+\mu\) in the coefficients of the quadratic equations of \(p\). Thus the critical exponent for \eqref{equ:original} should be the positive root of
	\begin{equation*}
		(n+\mu-1)p^2-(n+\mu+1)p-2=0,
	\end{equation*}
	while the conformal exponent for \eqref{equ:original} should be
	\begin{equation}\label{pconf}
		p_{\conf}(n,\mu)=\frac{n+\mu+3}{n+\mu-1}.
	\end{equation}
	In \cite{HSZ}, the regular Cauchy problem of \eqref{equ:original} was studied and the global existence result for small data solution for \(n\geq3\), \(\mu\in(0,1)\cup(1,2)\) and \(p_{\crit}(n,\mu)<p\leq p_{\conf}(n,\mu)\) was obtained.
	
	In this article, we will apply the Morawetz type energy estimate \eqref{equ:Morawetz} established in Theorem \ref{thm:Morawetz} to study the remaining case \(\mu=1\). With the Morawetz type energy estimate \eqref{equ:Morawetz} in hand, the weighted \(L^2-L^2\) estimate for the solution of \eqref{equ:linear-inhomo} follows immediately. While the weighted \(L^1-L^\infty\) estimate can be obtained by writing out the solution \(\phi\) of  \eqref{equ:linear-inhomo} in terms of the Fourier integral operators.
	(Here we emphasize that for the case $\mu=1$, the amplitude function of the Fourier integral operators admits an extra singularity of logarithm order as the frequency variable approaches to zero (see \eqref{equ:z-small-1}). Hence the case \(\mu=1\) is more difficult than the case \(\mu\in(0,1)\cup(1,2)\)). Then an interpolation between the weighted \(L^2-L^2\) estimate and weighted \(L^1-L^\infty\) estimate gives the weighted Strichartz estimates of \eqref{equ:linear-inhomo} for relatively large time. For the relatively small time, we apply the Morawetz type energy estimate \eqref{equ:Morawetz} to get weighted Strichartz estimates of \eqref{equ:linear-inhomo} under the condition \(\textup{supp}\phi\subseteq\left\{|x|\leq \frac{3t}{4}\right\}\), by the observation that \eqref{equ:linear-inhomo} is essentially an ultra-hyperbolic equation in \(n+2\) (since $\mu=1$) spacetime dimension, we then get the corresponding weighted Strichartz estimates in $\R^{n+2}$, where we can introduce a suitable transform similar to the Lorentz rotation to remove the support restriction for \(\phi\). Backing to the original spacetime \(\R^{n+1}_+\), then the expected Strichartz estimates for relatively small time are established. Collecting these estimates together, the global existence of small data solution follows by a standard Picard iteration. 
	
	The second main result in this article can be stated as follows.
	
	\begin{theorem}\label{YH-1}
		Assume that \(n\geq4\) and \(p_{\crit}(n,1)<p\leq p_{\conf}(n,1)\) and consider the following regular Cauchy problem
		\begin{equation}\label{equ:original-1}
			\left\{ \enspace
			\begin{aligned}
				&\partial_t^2 \phi-\Delta \phi +\f{1}{t}\,\p_t\phi=|\phi|^p,
				\quad (t,x)\in(2,\infty)\times\R^n, \\
				&\phi(2,x)=\varepsilon \phi_0(x), \quad \partial_{t} \phi(2,x)=\varepsilon \phi_1(x).
			\end{aligned}
			\right.
		\end{equation}
		Suppose $\phi_i\in C_c^{\infty}(\R^n)$ ($i=0, 1$),  then for $\ve>0$ small enough the Cauchy problem \eqref{equ:original-1} admits a global weak solution \(\phi\) with
		\begin{equation*}
			\left(1+\big|t^2-|x|^2\big|\right)^{\gamma}
			t^{\frac{1}{p+1}}\phi\in L^{p+1}([2,\infty)\times\R^{n}),
		\end{equation*}
		for some \(\gamma\) satifying
		\[\frac{1}{p(p+1)}<\gamma<\frac{np-(n+2)}{2(p+1)}. \]
	\end{theorem}
	
	This article is organized as follows: We establish the Morawetz energy estimate \eqref{equ:Morawetz} and hence prove Theorem \ref{thm:Morawetz} in Section 2. In Section 3 we obtain the weighted Strichartz estimates for the linear Cauchy problem. Finally some necessary lemmas are provided in the Appendix.
	
	\vspace{4mm}
	\noindent\textbf{Notations~}
	
	Throughout the paper, we use \(C\) to denote an absolute constant that may vary from line to line. The notation
	\(A\lesssim B\) means that \(A\leq CB\) for some absolute constant \(C>0\). 
	
	We denote by \(\delta, \delta_1, \delta_2\)  positive constants which may be made arbitrarily small.
	
	We also fix a smooth cutoff function \(\rho\in C^\infty_c(\R^n)\) satisfying
	\begin{equation}\label{equ:cut-off}
		\textup{supp}\rho\subseteq\left\{\xi\mid\frac{1}{2}\leq|\xi|\leq2\right\},
		\quad  \ds \sum\limits_{j=-\infty}^\infty\rho\left(2^{-j}\xi\right) \equiv1\quad \text{for} \quad \xi\in\R^n,
	\end{equation}
	and set \(\rho_j(\xi)\) as \(\rho(2^{-j}\xi)\).
	
	\section{Proof of Morawetz Estimate}
	
	In this section, we demonstrate the proof of the Morawetz energy estimate \eqref{equ:Morawetz} in Theorem \ref{thm:Morawetz}. The first step is to prove the following lemma.
	\begin{lemma}\label{Mora}
		Let $\phi$ solves the Cauchy problem \eqref{equ:linear-inhomo} with $n\ge 4$ and $t_0\ge 0$, then for $\mu\in (0, 2)$ and $\gamma>1$ it holds
		\begin{align}\label{007}
			& \underset{t_0< t\leq T}{\textup{sup}} t^{\frac{\mu}{2}} \left[  \left\| |u|^{\frac{1}{2}} \nabla_{t,x} \phi \right\| _{L^2 (\mathbb R^n)} +\left \| |u|^{\frac{1}{2}}  \frac{\phi}{r} \right\|  _{L^2 (\mathbb R^n)}+ \left\|  |u|^{\frac{1}{2}} \frac{\phi}{t}   \right\|_{L^2 (\mathbb R^n)} \right] \nonumber\\
			\le &C\left\| t ^{\frac{\mu}{2}} \left(  1+ |u|  \right)^{\frac{\gamma}{2}} \left(  1+ \underline{u}  \right)^{\frac{\gamma}{2}} F \right\|  _{L^2 ([t_0,T] \times \mathbb R^n  )},
		\end{align}
		where the positive constant $C$ depends on $n, \mu$ and $\gamma$.
	\end{lemma}
	\begin{proof}[Proof of Lemma \ref{Mora}]
		We will use the following multiplier
		\begin{equation}\label{m1}
			X= \left(t+r\right) \left(\partial_t + \partial_r + \frac{n-1}{2r} + \frac{\mu}{2t}\right) + |t-r| \left(\partial_t - \partial_r - \frac{n-1}{2r} + \frac{\mu}{2t}\right)	
		\end{equation}
		to the inhomogeneous wave equation
		\begin{equation*}
			\Box\phi + \frac{\mu}{t}\partial_t\phi = \partial^2_t\phi - \partial^2_r\phi - \frac{n-1}{r}\partial_r\phi - \frac{1}{r^2} \Delta_{S^{n-1}} \phi  + \frac{\mu}{t}\partial_t\phi = F.
		\end{equation*}
		For the first part corresponding to the multiplier \eqref{m1} we have
		\begin{align*}
			& (t+r)\left(\partial_t+\partial_r + \frac{n-1}{2r} + \frac{\mu}{2t}\right)\phi \left(\Box\phi + \frac{\mu}{t}\partial_t\phi\right)r^{n-1}t^{\mu} \nonumber\\
			& =(t+r) \left(\partial_t\phi+\partial_r\phi + \frac{n-1}{2r}\phi + \frac{\mu}{2t}\phi\right) r^{\frac{n-1}{2}}t^{\frac{\mu}{2}} \nonumber\\
			& \quad \times\left( \partial^2_t\phi -\partial^2_r\phi - \frac{n-1}{r}\partial_r\phi - \frac{1}{r^2} \Delta_{S^{n-1}} \phi +\frac{\mu}{t} \partial_t \phi\right) r^{\frac{n-1}{2}} t ^{\frac{\mu}{2}} \nonumber\\
			& = (t+r)  \left(\partial_t+\partial_r\right) \left(r^{\frac{n-1}{2}} t ^{\frac{\mu}{2}} \phi   \right)\big[ (\partial_t-\partial_r)(\partial_t+\partial_r) \left(r^{\frac{n-1}{2}} t^{\frac{\mu}{2}} \phi  \right) \nonumber\\
			&  \quad + \frac{(n-1)(n-3)}{4}r^{\frac{n-1}{2}-2} t^{\frac{\mu}{2}}\phi - \frac{\mu}{2}\left(\frac{\mu}{2}-1\right)r^{\frac{n-1}{2}}t^{\frac{\mu}{2}-2}\phi
			- r^{\frac{n-1}{2}-2} t^{\frac{\mu}{2}} \Delta_{S^{n-1}} \phi \big] \nonumber\\
			& = (\partial_t-\partial_r) \left\{ \frac{t+r}{2} \left[ (\partial_t+\partial_r)  \left(r^{\frac{n-1}{2}} t^{\frac{\mu}{2}} \phi \right)  \right]^2   \right\}  \nonumber\\
			&\quad +  (\partial_t+\partial_r) \left\{\frac{t+r}{8}\left[\frac{(n-1)(n-3)}{r^2}-\frac{\mu(\mu-2)}{t^2}\right]\left( r^{\frac{n-1}{2}} t^{\frac{\mu}{2}} \phi  \right)^2\right\} \nonumber\\
			&\quad
			+  \frac{(n-1)(n-3)t}{4r^3} \left( r^{\frac{n-1}{2}} t ^{\frac{\mu}{2}} \phi \right)^2 - \frac{\mu(\mu-2) r}{4t^3} \left( r^{\frac{n-1}{2}} t ^{\frac{\mu}{2}} \phi \right)^2\nonumber\\
			&\quad
			-\nabla_{S^{n-1}}\left[ \frac{(t+r)}{r^{2}}\left(\partial_t + \partial_r \right)\left(r^{\frac{n-1}{2}}t^{\frac{\mu}{2}} \phi\right) \nabla_{S^{n-1}}\left(r^{\frac{n-1}{2}}t^{\frac{\mu}{2}}\phi\right) \right] \nonumber\\
			& \quad +(\partial_t+\partial_r) \left\{ \frac{t+r}{2r^{2}} \left[ \nabla_{S^{n-1}}\left(r^{\frac{n-1}{2}} t^{\frac{\mu}{2}} \phi \right)  \right]^2   \right\} + \frac{t}{r^{3}} \left[ \nabla_{S^{n-1}}\left(r^{\frac{n-1}{2}} t^{\frac{\mu}{2}} \phi \right)  \right]^2,
		\end{align*}
		while for the second part we have
		\begin{align}\label{01}
			& \left| t - r \right|\left(\partial_t-\partial_r-\frac{n-1}{2r} + \frac{\mu}{2t}\right)\phi \left(\Box\phi + \frac{\mu}{t}\partial_t\phi\right)r^{n-1}t^{\mu} \nonumber\\
			& =  \left| t - r \right| \left(  \partial_t\phi-\partial_r\phi-\frac{n-1}{2r}\phi + \frac{\mu}{2t}\phi   \right) r^{\frac{n-1}{2}} t ^{\frac{\mu}{2}}  \nonumber\\
			& \times\left( \partial^2_t\phi -\partial^2_r\phi - \frac{n-1}{r}\partial_r\phi - \frac{1}{r^2} \Delta_{S^{n-1}} \phi +\frac{\mu}{t} \partial_t \phi\right) r^{\frac{n-1}{2}} t ^{\frac{\mu}{2}} \nonumber\\
			& = \left(\partial_t+\partial_r\right)   \left\{  \frac{\left| t - r \right|}{2} \left[ (\partial_t-\partial_r) \left(r^{\frac{n-1}{2}} t^{\frac{\mu}{2}} \phi  \right) \right]^2  \right\} \nonumber\\
			& \quad + \left(\partial_t-\partial_r\right)  \left\{\frac{|t-r|}{8}\left[\frac{(n-1)(n-3)}{r^2}-\frac{\mu(\mu-2)}{t^2}\right]\left( r^{\frac{n-1}{2}} t^{\frac{\mu}{2}} \phi  \right)^2\right\} \nonumber\\
			& \quad - \frac{\mathrm {sgn} (t-r) r + \left| t - r \right|}{r^3} \Big[  \frac{(n-1)(n-3)}{4}  \left( r^{\frac{n-1}{2}} t ^{\frac{\mu}{2}} \phi \right)^2+\left( \nabla_{S^{n-1}} \left( r^{\frac{n-1}{2}} t ^{\frac{\mu}{2}} \phi \right) \right)^2   \Big] \nonumber\\
			& \quad + \frac{ \left[\mathrm {sgn} (t-r) t- \left| t - r \right| \right] \mu (\mu-2) }{4t^3} \left( r^{\frac{n-1}{2}} t ^{\frac{\mu}{2}} \phi \right)^2 \nonumber\\
			& \quad - \nabla_{S^{n-1}}  \left[ \frac{\left| t - r \right|}{r^2} (\partial_t-\partial_r) \left( r^{\frac{n-1}{2}} t ^{\frac{\mu}{2}} \phi \right) \nabla_{S^{n-1}}  \left( r^{\frac{n-1}{2}} t ^{\frac{\mu}{2}} \phi \right)   \right].
		\end{align}

		There are two cases due to the absolute vale sign appearing in \eqref{01}: \\
		(1): If $ t\geq r $, then
		\begin{align*}
			& \mathrm {sgn} (t-r) r + \left| t - r \right| = r+t-r=t, \nonumber\\
			&  \mathrm {sgn} (t-r) t- \left| t - r \right| = t-t+r=r
		\end{align*}
		and hence
		\begin{align*}
			& \left( \frac{t}{r^3}-\frac{\mathrm {sgn} (t-r) r + \left| t - r \right|   }{r^3} \right) \Big[   \frac{(n-1)(n-3)}{4}  \left( r^{\frac{n-1}{2}} t ^{\frac{\mu}{2}} \phi \right)^2+\left( \nabla_{S^{n-1}} \left( r^{\frac{n-1}{2}} t ^{\frac{\mu}{2}} \phi \right) \right)^2  \Big] = 0,
		\end{align*}
		\begin{align*}
			&  - \frac{\mu(\mu-2) r}{4t^3} \left( r^{\frac{n-1}{2}} t ^{\frac{\mu}{2}} \phi \right)^2+\frac{ \left[\mathrm {sgn} (t-r) t - \left| t - r \right| \right] \mu (\mu-2) }{4t^3} \left( r^{\frac{n-1}{2}} t ^{\frac{\mu}{2}} \phi \right)^2  = 0.
		\end{align*}
		(2): If $ t < r $, then
		\begin{align*}
			& \mathrm {sgn} (t-r) r + \left| t - r \right| = -r+r-t=-t, \nonumber\\
			&  \mathrm {sgn} (t-r) t- \left| t - r \right| = -t - (r-t)=-r
		\end{align*}
		and hence
		\begin{align*}
			& \left( \frac{t}{r^3}-\frac{\mathrm {sgn} (t-r) r + \left| t - r \right|   }{r^3} \right) \Big[   \frac{(n-1)(n-3)}{4}  \left( r^{\frac{n-1}{2}} t ^{\frac{\mu}{2}} \phi \right)^2+  \left( \nabla_{S^{n-1}} \left( r^{\frac{n-1}{2}} t ^{\frac{\mu}{2}} \phi \right) \right)^2  \Big]  \nonumber\\
			& = \frac{2t}{r^3}  \Big[   \frac{(n-1)(n-3)}{4}  \left( r^{\frac{n-1}{2}} t ^{\frac{\mu}{2}} \phi \right)^2 +  \left( \nabla_{S^{n-1}} \left( r^{\frac{n-1}{2}} t ^{\frac{\mu}{2}} \phi \right) \right)^2  \Big]  \geq 0,
		\end{align*}
		\begin{align*}
			& - \frac{ \mu (\mu-2) }{4t^3}\left[ r- \left( \mathrm {sgn} (t-r) t - \left| t - r \right| \right) \right] \left( r^{\frac{n-1}{2}} t ^{\frac{\mu}{2}} \phi \right)^2= - \frac{ \mu (\mu-2) }{4t^3} 2r\left( r^{\frac{n-1}{2}} t ^{\frac{\mu}{2}} \phi \right)^2 \geq 0
		\end{align*}
		for $ 0< \mu < 2 $. So we have
		\begin{align*}
			& X\phi \left(\Box\phi + \frac{\mu}{t}\partial_t\phi  \right) r^{n-1} t^{\mu} = X\phi F r^{n-1} t^{\mu} \nonumber\\
			&  \geq  (\partial_t-\partial_r) \left\{ \frac{t+r}{2} \left[ (\partial_t+\partial_r)  \left(r^{\frac{n-1}{2}} t^{\frac{\mu}{2}} \phi \right)  \right]^2   \right\}  \nonumber\\
			&\quad  +  (\partial_t+\partial_r) \left\{\frac{t+r}{8}\left[\frac{(n-1)(n-3)}{r^2}-\frac{\mu(\mu-2)}{t^2}\right]\left( r^{\frac{n-1}{2}} t^{\frac{\mu}{2}} \phi  \right)^2\right\} \nonumber\\
			&\quad
			-\nabla_{S^{n-1}}\left[ \frac{(t+r)}{r^{2}}\left(\partial_t + \partial_r \right)\left(r^{\frac{n-1}{2}}t^{\frac{\mu}{2}} \phi\right) \nabla_{S^{n-1}}\left(r^{\frac{n-1}{2}}t^{\frac{\mu}{2}}\phi\right) \right] \nonumber\\
			&\quad  +(\partial_t+\partial_r) \left\{ \frac{t+r}{2r^{2}} \left[ \nabla_{S^{n-1}}\left(r^{\frac{n-1}{2}} t^{\frac{\mu}{2}} \phi \right)  \right]^2+\frac{\left| t - r \right|}{2} \left[ (\partial_t-\partial_r) \left(r^{\frac{n-1}{2}} t^{\frac{\mu}{2}} \phi  \right) \right]^2   \right\} \nonumber\\
			& \quad + \left(\partial_t-\partial_r\right) \left\{\frac{|t-r|}{8}\left[\frac{(n-1)(n-3)}{r^2}-\frac{\mu(\mu-2)}{t^2}\right]\left( r^{\frac{n-1}{2}} t^{\frac{\mu}{2}} \phi  \right)^2\right\}  \nonumber\\
			&\quad - \nabla_{S^{n-1}}  \Big[ \frac{\left| t - r \right|}{r^2} (\partial_t-\partial_r) \left( r^{\frac{n-1}{2}} t ^{\frac{\mu}{2}} \phi \right)\nabla_{S^{n-1}}  \left( r^{\frac{n-1}{2}} t ^{\frac{\mu}{2}} \phi \right) \Big].
		\end{align*}
		Integrating the above inequality over the time slice $ t=C $, $ u=t-r=C $ and $ \underline{u}= t+r=C $, we get
		\begin{align}\label{03}
			& \underset{t}{\text{sup}} \int_{\mathbb R^3} \Bigg[ \frac{t+r}{2} t^{\mu} \left( \phi_t + \phi_r + \frac{n-1}{2} \frac{\phi}{r} + \frac{\mu}{2} \frac{\phi}{t}  \right)^2+ \frac{\left| t - r \right|}{2} t^{\mu} \left( \phi_t - \phi_r - \frac{n-1}{2} \frac{\phi}{r} + \frac{\mu}{2} \frac{\phi}{t}  \right)^2 \nonumber\\
			&\quad +  \left(\frac{(n-1)(n-3) }{8r^2}-\frac{\mu (\mu-2)}{8t^2}\right) \left(  t+r+ \left| t - r \right|  \right) t^{\mu} \phi^2 + \frac{t+r+\left| t - r \right|}{2r^2} t^{\mu} (\nabla_{S^{n-1}}\phi)^2 \Bigg] \mathrm{d}x \nonumber\\
			&\quad + \underset{u}{\text{sup}} \int_{t-r=u} \int_{S^{n-1}} (t+r) \left( \phi_t + \phi_r + \frac{n-1}{2} \frac{\phi}{r} + \frac{\mu}{2} \frac{\phi}{t}  \right)^2 r^{n-1} t^{\mu} \mathrm{d} \omega \mathrm{d}r \nonumber\\
			&\quad +2 \underset{u}{\text{sup}} \int_{t-r=u} \int_{S^{n-1}} \Bigg\{ \left(\frac{(n-1)(n-3) }{8r^2}-\frac{\mu (\mu-2)}{8t^2}\right) \left| t - r \right|  \left(  r^{\frac{n-1}{2}} t ^{\frac{\mu}{2}} \phi   \right)^2 \nonumber\\
			& \quad+ \frac{\left| t - r \right| }{2r^2} \left[ \nabla_{S^{n-1}} \left(   r^{\frac{n-1}{2}} t ^{\frac{\mu}{2}} \phi  \right)  \right]^2  \Bigg\} \mathrm{d} \omega \mathrm{d}r \nonumber\\
			&\quad + \underset{\underline{u}}{\text{sup}} \int_{t+r=\underline{u}} \int_{S^{n-1}} \left| t - r \right|  \left( \phi_t - \phi_r - \frac{n-1}{2} \frac{\phi}{r} + \frac{\mu}{2} \frac{\phi}{t}  \right)^2 \mathrm{d} \omega \mathrm{d}r \nonumber\\
			& \quad+ 2 \underset{\underline{u}}{\text{sup}} \int_{t+r=\underline{u}} \int_{S^{n-1}}
			\Bigg\{\left(\frac{(n-1)(n-3) }{8r^2}-\frac{\mu (\mu-2)}{8t^2}\right) (t+r) \left(  r^{\frac{n-1}{2}} t ^{\frac{\mu}{2}} \phi   \right)^2 \nonumber\\
			& \quad+\frac{t+r }{2r^2} \left[ \nabla_{S^{n-1}} \left(   r^{\frac{n-1}{2}} t ^{\frac{\mu}{2}} \phi  \right)  \right]^2  \Bigg\}\mathrm{d} \omega \mathrm{d}r \nonumber\\
			& \leq \iiint\Bigg[  (t+r)  \left( \phi_t + \phi_r + \frac{n-1}{2r} \phi + \frac{\mu}{2t} \phi  \right)+ \left| t - r \right| \left( \phi_t - \phi_r - \frac{n-1}{2r} \phi + \frac{\mu}{2t} \phi  \right)\Bigg]F r^{n-1} t^{\mu} \mathrm{d} \omega \mathrm{d}r \mathrm{d}t.
		\end{align}
		It is easy to get
		\begin{align*}
			& \iiint(t+r)  \left( \phi_t + \phi_r + \frac{n-1}{2r} \phi + \frac{\mu}{2t} \phi  \right)  F r^{n-1} t^{\mu} \mathrm{d} \omega \mathrm{d}r \mathrm{d}t \nonumber\\
			& \lesssim  \underset{u}{\text{sup}}  \left(  \int_{\underline{u} = t+r} \int (t+r)  \left( \phi_t + \phi_r + \frac{n-1}{2r} \phi + \frac{\mu}{2t} \phi  \right)^2 r^{n-1} t^{\mu} \mathrm{d} \omega \mathrm{d}r \right) ^{\frac{1}{2}} \nonumber\\
			& \quad  \times \int  \left( \int_{\underline{u} = t+r}  \int (1+ \underline{u}) F^2  t^{\mu} r^{n-1} \mathrm{d}\theta\mathrm{d}\underline{u} \right) ^{\frac{1}{2}} \mathrm{d}u \nonumber\\
			& \lesssim  \underset{u}{\text{sup}}  \left(  \int_{\underline{u} = t+r} \int (t+r)  \left( \phi_t + \phi_r + \frac{n-1}{2r} \phi + \frac{\mu}{2t} \phi  \right)^2 r^{n-1} t^{\mu} \mathrm{d} \omega \mathrm{d}r \right) ^{\frac{1}{2}} \nonumber\\
			& \quad  \times \int  \left( 1+ |u| \right)^{-\frac{\gamma}{2}}\left( 1+ |u| \right)^{\frac{\gamma}{2}} \left(  \int_{\underline{u} = t+r}\int (1+\underline{u}) F^2  t^{\mu} r^{n-1} \mathrm{d}\theta\mathrm{d}\underline{u} \right) ^{\frac{1}{2}} \mathrm{d}u \nonumber\\
			& \lesssim  \underset{u}{\text{sup}}  \left(  \int_{\underline{u} = t+r} \int (t+r)  \left( \phi_t + \phi_r + \frac{n-1}{2r} \phi + \frac{\mu}{2t} \phi  \right)^2 r^{n-1} t^{\mu} \mathrm{d} \omega \mathrm{d}r \right) ^{\frac{1}{2}} \nonumber\\
			& \quad  \times  \left(  \int_{u}  \left(  1+ |u|  \right)^{-\gamma} \mathrm{d} u \right) ^{\frac{1}{2}} \left( \iiint\left(  1+ |u|  \right)^{\gamma} \left(  1+ \underline{u}  \right)^{\gamma} F^2  t^{\mu} r^{n-1} \mathrm{d}\theta\mathrm{d}\underline{u}\mathrm{d}u \right) ^{\frac{1}{2}} \nonumber\\
			& \lesssim  \underset{u}{\text{sup}}  \left(  \int_{\underline{u} = t+r} \int (t+r)  \left( \phi_t + \phi_r + \frac{n-1}{2r} \phi + \frac{\mu}{2t} \phi  \right)^2 r^{n-1} t^{\mu} \mathrm{d} \omega \mathrm{d}r \right) ^{\frac{1}{2}} \nonumber\\
			& \quad  \times C \left( \iiint\left(  1+ |u|  \right)^{\gamma} \left(  1+ \underline{u}  \right)^{\gamma} F^2  t^{\mu} r^{n-1} \mathrm{d}\theta\mathrm{d}\underline{u}\mathrm{d}u \right) ^{\frac{1}{2}}.
		\end{align*}
		In a similar way,
		\begin{align*}
			& \int \int \int  \left| t - r \right|  \left( \phi_t - \phi_r - \frac{n-1}{2r} \phi + \frac{\mu}{2t} \phi  \right)  F r^{n-1} t^{\mu} \mathrm{d} \omega \mathrm{d}r \mathrm{d}t \nonumber\\
			& \lesssim  \underset{\underline{u}}{\text{sup}}  \left(  \int_{u = t-r} \int \left| t - r \right|  \left( \phi_t - \phi_r -\frac{n-1}{2r} \phi + \frac{\mu}{2t} \phi  \right)^2 r^{n-1} t^{\mu} \mathrm{d} \omega \mathrm{d}r \right) ^{\frac{1}{2}} \nonumber\\
			& \quad  \times C \left( \int\int\int \left(  1+ |u|  \right)^{\gamma} \left(  1+ \underline{u}  \right)^{\gamma} F^2  t^{\mu} r^{n-1} \mathrm{d}\theta\mathrm{d}\underline{u}\mathrm{d}u \right) ^{\frac{1}{2}}.
		\end{align*}
		On the other hand, it holds that
		\begin{align}\label{06}
			&\frac{t+r}{2} \left( \phi_t + \phi_r + \frac{n-1}{2} \frac{\phi}{r} + \frac{\mu}{2} \frac{\phi}{t}  \right)^2 t^{\mu}+ \frac{\left| t - r \right|}{2} \left( \phi_t - \phi_r - \frac{n-1}{2} \frac{\phi}{r} + \frac{\mu}{2} \frac{\phi}{t}  \right)^2 t^{\mu} \nonumber\\
			& \quad+  \left(\frac{(n-1)(n-3) }{8r^2}-\frac{\mu (\mu-2)}{8t^2}\right) \left(  t+r+ \left| t - r \right| \right) t^{\mu} \phi^2 \nonumber\\
			& \gtrsim \left| t - r \right| \left[ \left( \phi_t  + \frac{\mu}{2} \frac{\phi}{t} \right)^2 + \left( \phi_r + \frac{n-1}{2} \frac{\phi}{r} \right)^2 \right] t^{\mu} \nonumber+  \left| t - r \right|  \left(\frac{(n-1)(n-3) }{4r^2}- \frac{\mu (\mu-2)}{4t^2}\right)t^{\mu} \phi^2\nonumber\\
			& \gtrsim \left| t - r \right|t^{\mu} \left[ \phi_t^2 + \left( \frac{\phi}{t}  \right)^2 + \phi_r^2 + \left( \frac{\phi}{r}  \right)^2 \right],
		\end{align}
		which leads to for any $T>2  $, by combining \eqref{03}-\eqref{06}
		\begin{align*}
			& \underset{t_0< t\leq T}{\text{sup}} t^{\frac{\mu}{2}} \left[  \left\| \left(  |t-r|  \right)^{\frac{1}{2}} \nabla_{t,x} \phi \right\| _{L^2 (\mathbb R^n)} +\left \| \left(  |t-r|  \right)^{\frac{1}{2}}  \frac{\phi}{r} \right\|  _{L^2 (\mathbb R^n)}+ \left\|  \left(  |t-r|  \right)^{\frac{1}{2}} \frac{\phi}{t}   \right\|_{L^2 (\mathbb R^n)} \right] \nonumber\\
			\lesssim &\left\| t ^{\frac{\mu}{2}} \left(  1+ |u|  \right)^{\frac{\gamma}{2}} \left(  1+ \underline{u}  \right)^{\frac{\gamma}{2}} F\right \|  _{L^2 ([t_0,T] \times \mathbb R^n  )},~~~ \gamma>1,
		\end{align*}
		which proves \eqref{007}.
	\end{proof}
	
	With Lemma \ref{Mora} in hand, in order to prove \eqref{equ:Morawetz} in Theorem \ref{thm:Morawetz}, it suffices to prove 
	\begin{lemma}\label{lem:plus1}
		Consider the Cauchy problem for the linear wave equation \eqref{equ:linear-inhomo} with \(n\geq4\) and \(t_0\geq0\). Then for \(\mu\in(0,2)\) it holds
		\begin{align}\label{08}
			& \underset{t_0< t\leq T}{\textup{sup}} t^{\frac{\mu}{2}} \left[  \left\|  \nabla_{t,x} \phi \right\| _{L^2 (\mathbb R^n)} +\left \| \frac{\phi}{r} \right\|  _{L^2 (\mathbb R^n)}+ \left\|   \frac{\phi}{t}   \right\|_{L^2 (\mathbb R^n)} \right] \nonumber\\
			\lesssim& \left\| \tau ^{\frac{\mu}{2}} \left(  1+ |u|  \right)^{\frac{\gamma}{2}} \left(  1+ \underline{u}  \right)^{\frac{\gamma}{2}} F \right\|  _{L^2 ([t_0,T] \times \mathbb R^n  )},~~~ \gamma>1.
		\end{align}
	\end{lemma}
	\begin{proof}[Proof of Lemma \ref{lem:plus1}]
		By Hardy inequality, we have
		\[\lp\frac{\phi(t)}{r}\rp_{L^2(\mathbb R^n)}\leq C\left\|\nabla_x\phi\right\|_{L^2(\mathbb R^n)}.\]
		Also it holds
		\begin{equation*}
			\begin{split}
				\sup_{t_0\leq t\leq T}
				\lp\frac{\phi(t)}{t}\rp_{L^2(\mathbb R^n)}&\leq \sup_{t_0\leq t\leq T}
				\lp\frac{1}{t}\int_{t_0}^t\partial_\tau\phi(\tau)\md \tau\rp_{L^2(\mathbb R^n)}   \\
				&\leq\sup_{t_0\leq t\leq T}\frac{1}{t}
				\int_{t_0}^t\lp\partial_\tau\phi(\tau)
				\rp_{L^2(\mathbb R^n)}\md \tau \\
				&\leq\sup_{t_0\leq t\leq T}
				\lp\partial_t\phi(t)\rp_{L^2(\mathbb R^n)}. \\
			\end{split}
		\end{equation*}
		Hence \eqref{08} follows immediately from Lemma~\ref{lem:energy} and Lemma~\ref{lem:energy-t}.
	\end{proof}

	\section{Estimates for the linear problem}\label{linear-estimate}
	
	In this section we establish the weighted Strichartz estimates which are necessary for the global existence of the small data solution to the Cauchy problem \eqref{equ:original-1}.
	
	\subsection{Weighted Strichartz estimate for linear homogeneous equation}
	We firstly consider for \(t_0\geq0\) the corresponding Cauchy problem of the linear homogeneous equation
	\begin{equation}\label{equ:homo}
		\left\{ \enspace
		\begin{aligned}
			&\partial_t^2 v-\Delta v +\f{\mu}{t}\,\p_tv=0,
			\quad (t,x)\in(t_0,\infty)\times\R^n, \\
			&u(t_0,x)=\varepsilon v_0(x), \quad \partial_{t} v(t_0,x)=\varepsilon v_1(x).
		\end{aligned}
		\right.
	\end{equation}
	By the work of Wirth \cite[Theorem 2.1]{Wirth-0}, we have
	\[v(t,x)=\Phi_0(t,t_0,D)v_0+\Phi_1(t,t_0,D)v_1,\]
	where the multipliers \(\Phi_j\) are given by
	\begin{equation}\label{equ:Phi-0}
		\Phi_0(t,2,\xi)=\frac{i\pi}{4}|\xi|\frac{t^\nu}{2^{\nu-1}}\left[H_{\nu-1}^-(2|\xi|)H_\nu^+(t|\xi|)
		-H_\nu^+(2|\xi|)H_{\nu-1}^-(t|\xi|)\right],    
	\end{equation}
	and
	\begin{equation}\label{equ:Phi-1}
		\Phi_1(t,2,\xi)=-\frac{i\pi}{4}\frac{t^\nu}{2^\nu}\left[H_\nu^-(2|\xi|)
		H_\nu^+(t|\xi|)-H_\nu^+(2|\xi|)H_\nu^-(t|\xi|)\right], 
	\end{equation}
	where \(H_\nu^\pm(z)\) are the pairs of Hankel function of \(\nu\) order with \(\nu=\frac{1-\mu}{2}\). 
	By \cite[Proposition 3.1]{Wirth-0}, for the functions \(H^\pm_\nu\), there exists a constant \(K>0\) such that
	\begin{equation}\label{equ:z-large-1}
		H_\nu^\pm(z)=e^{\pm iz}a_{\pm}(z), \quad a_{\pm}(z)\in S^{-\frac{1}{2}}([K,+\infty)),
	\end{equation}
	where we denote by \(S^{-\frac{1}{2}}\) the space of classical symbols of order \(-\frac{1}{2}\).
	While for small arguments, \(0<z\leq c<1\), we have
	\begin{equation}\label{equ:z-small-1}
		|H_\nu^\pm(z)|\lesssim
		\left\{ \enspace
		\begin{aligned}
			&z^{-|\nu|}, && \text{if} \quad \nu\neq0 \\
			&-\ln|z|, && \text{if} \quad \nu=0.
		\end{aligned}
		\right.
	\end{equation}
	We conclude from \eqref{equ:z-small-1} that if \(z\rightarrow0\), then \(H_\nu^\pm(z)\) admits an extra singularity of \(-\ln|z|\) provided \(\nu=0\) (or equivalently \(\mu=1\)). Thus the analysis for \(\mu=1\) is the most difficult case. On the other hand, in order to prove Theorem~\ref{YH-1}, it suffices to consider \(\mu=1\) and \(t_0=2\), thus in the remaining part of Section~\ref{linear-estimate}, we will focus on \(\mu=1\) and \(t_0=2\).  The first result is the following lemma.
	
	\begin{lemma}\label{lem:homogeneous}
		If \(n\geq2\), \(q>2\),  \(\gamma<\frac{n}{2}-\frac{n+1}{q}\) and \(v_0, v_1\in C^\infty(\R^n)\) with \(\textup{supp}(v_0,v_1)\subseteq\{x\mid |x|\leq1\}\), then it holds for the Cauchy problem \eqref{equ:homo}
		\begin{equation}\label{equ:homo-estimate}
			\left\|\big(t^2-|x|^2\big)^\gamma t^{\frac{1}{q}}v\right\|_{L^q([2,+\infty)\times\mathbb{R}^n)} \leq \ve C(v_0,v_1).
		\end{equation}
	\end{lemma}
	\begin{proof}[Proof of Lemma~\ref{lem:homogeneous}]
		The proof of \eqref{equ:homo-estimate} is based on establishing a suitable pointwise estimate of \(v\). Recall that for the free wave equation
		\(\Box\tilde{v}=0\) with initial data as in Lemma~\ref{lem:homogeneous}, we have
		\begin{equation}\label{equ:free-wave}
			\tilde{v}=O\big(\ve(1+t)^{-\frac{n-1}{2}}(1+|t-|x||)^{-\frac{n-1}{2}}\big),
		\end{equation}
		and since the damping term \(\partial_tv/t\) provides more decay, we have that \eqref{equ:free-wave} holds for the solution \(v\) of \eqref{equ:homo}. Moreover, for any fixed \(T_0\gg1\) and \(t\in[2,T_0]\), the support condition of the initial data implies \(1\leq t-|x|\leq t\leq T_0\) for \((t,x)\in\textup{supp}v\), thus we have
		\begin{equation}\label{equ:homo-finit}
			|v|\leq \ve C(v_0,v_1,T_0)(1+t)^{-\frac{n}{2}}(1+|t-|x||)^{-\frac{n}{2}},\quad \textup{for}\quad 2\leq t\leq T_0.
		\end{equation}
		
		By \eqref{equ:homo-finit}, it suffice to establish the pointwise estimate for \(v\) with \(t\geq T_0\gg1\). \eqref{equ:Phi-0}-\eqref{equ:Phi-1} imply that the properties of \(\Phi_0\) and \(\Phi_1\) are similar, thus it suffices to consider \(\Phi_1(t,t_0,D)v_1\). Furthermore,  \eqref{equ:z-large-1} and \eqref{equ:z-small-1} show that the behavior of \(H_0^+\) and \(H_0^-\) are similar, hence we only need to estimate \(H_0^-(2|\xi|)H_0^+(t|\xi|)\hat{v}_1(t,\xi)\).
		Motivated by \eqref{equ:z-small-1} and \eqref{equ:z-large-1}, we split the space of \(\xi\) into three parts:
		\begin{itemize}
			\item low frequency: \(t|\xi|\leq1\) or \(|\xi|\leq1/t\);
			\item medium frequency: \(2|\xi|\leq1\leq t|\xi|\) or \(1/t\leq|\xi|\leq1/2\);
			\item high frequency: \(2|\xi|\geq1\) or \(|\xi|\geq1/2\).
		\end{itemize}
		
		More specifically, define the Littlewood-Paley decomposition of solution $v$ by
		\[v_j(t,x)=\rho_j(D)H_0^-(2|D|)H_0^+(t|D|)v_1,\]
		with \(\rho_j\) defined in \eqref{equ:cut-off}. Then one has that for low frequency part
		\begin{equation}\label{equ:homo-low}
			\begin{split}
				|v_j| & =\left|\int_{\R^n}e^{ix\xi}\rho_j(\xi) H_0^-(2|\xi|)H_0^+(t|\xi|)\hat{v}_1(t,\xi)\md\xi\right|\\
				&\lesssim|j|\left|\ln(t2^j)|2^{nj}\|v_1(t,\cdot)\right\|_{L^1(\R^n)} \\
				&\lesssim|j|^22^{nj}(t2^j)^{-\frac{n}{2}}
				\left\|v_1(t,\cdot)\right\|_{L^1(\R^n)}\\
				&\lesssim2^{\frac{n-1}{2}j}t^{-\frac{n}{2}}
				\left\|v_1(t,\cdot)\right\|_{L^1(\R^n)}.
			\end{split}
		\end{equation}
		For the medium frequency part by stationary phase method
		\begin{equation}\label{equ:homo-medium}
			\begin{split}
				|v_j| & =\left|\int_{\R^n}e^{ix\xi}\rho_j(\xi) H_0^-(2|\xi|)e^{it|\xi|}a_+(t|\xi|)\hat{v}_1(t,\xi)\md\xi\right| \\
				& \lesssim|j|(t2^j)^{-\frac{1}{2}}
				2^{nj}(1+t2^j)^{-\frac{n-1}{2}}
				\|v_1(t,\cdot)\|_{L^1(\R^n)}\\
				& \lesssim|j|2^{nj}(t2^j)^{-\frac{n}{2}}
				\|v_1(t,\cdot)\|_{L^1(\R^n)}\\
				&\lesssim2^{\frac{n-1}{2}j}
				t^{-\frac{n}{2}}\|v_1(t,\cdot)\|_{L^1(\R^n)}.
			\end{split}
		\end{equation}
		While for the high frequency case
		\begin{equation*}
			\begin{split}
				|v_j| & =\left|\int_{\R^n}e^{ix\xi}\rho_j(\xi) e^{i(t-2)|\xi|}a_-(2|\xi|)a_+(t|\xi|)|\xi|^{-\frac{n}{2}}
				\widehat{|D|^{\frac{n}{2}}v}_1(t,\xi)\md\xi\right| \\
				& \lesssim2^{-\frac{1}{2}j}(t2^j)^{-\frac{1}{2}}  2^{nj}(1+t2^j)^{-\frac{n-1}{2}}
				2^{-\frac{n}{2}j}
				\|v_1(t,\cdot)\|_{\dot{W}^{\frac{n}{2},1}(\R^n)}\\
				& \lesssim2^{-\frac{1}{2}j}t^{-\frac{n}{2}}
				\|v_1(t,\cdot)\|_{W^{\frac{n}{2},1}(\R^n)}.
			\end{split}
		\end{equation*}
		Summing up with respect to \(j\), we get
		\[\left|H_0^-(2|D|)H_0^+(t|D|)v_1\right|\lesssim t^{-\frac{n}{2}}
		\|v_1(t,\cdot)\|_{W^{\frac{n}{2},1}(\R^n)},\]
		and hence
		\begin{equation}\label{equ:homo-infinite}
			|v(t,x)|\leq \ve C(v_0,v_1,T_0)(1+t)^{-\frac{n}{2}}.
		\end{equation}
		
		Comparing \eqref{equ:homo-infinite} with \eqref{equ:homo-finit}, we see that more decay in terms of \(1+|t-|x||\) is necessary. For this end, we estimate \(v(t,x)\) for different scales of \(|t-|x||\), respectively.
		\paragraph{Case I \(\mathbf{|t-|x||\leq10}\)}
		If \(|t-|x||\leq10\), we have by direct computation
		\begin{equation*}
			|v(t,x)|\leq \ve C(v_0,v_1,T_0)(1+t)^{-\frac{n}{2}}(1+|t-|x||)^{-\frac{n}{2}}.
		\end{equation*}
		\paragraph{Case II \(\mathbf{|t-|x||>10}\)}
		For the case \(|t-|x||>10\), the analysis is a little involved and we divide the proof into three parts according to the scale of the phase variable.
		\paragraph{Case II.1 Low frequency}
		First note that \(\textup{supp}v_1\subseteq B(0,1)\) and \(t\geq2\), thus by the finite speed of propagation, we have
		\begin{equation*}
			|x|\leq t-1, \quad \text{if} \quad (t,x)\in\textup{supp}v,
		\end{equation*}
		hence \(1\leq t-|x|\leq t\), and in the low frequency case we must have \(|t-|x||2^j\leq t2^j\leq1\), then \eqref{equ:homo-low} implies
		\begin{equation*}
			\begin{split}
				|v_j| & \lesssim|j|^22^{nj}(t2^j)^{-\frac{n}{2}}
				(|t-|x||2^j)^{-\frac{n}{2}+\delta}\|v_1(t,\cdot)\|_{L^1(\R^n)} \\
				&\lesssim2^{\delta j}(1+t)^{-\frac{n}{2}}
				(1+|t-|x||)^{-\frac{n}{2}+\delta}\|v_1(t,\cdot)\|_{L^1(\R^n)}.
			\end{split}
		\end{equation*}
		\paragraph{Case II.2 Medium frequency}
		For the medium frequency case, we apply \cite[(3.29)]{Gls2} together with \eqref{equ:homo-medium} to imply for any \(N>0\)
		\[|v_j|\leq2^{\frac{n-1}{2}j}t^{-\frac{n}{2}}
		\int_{\R^n}\big(1+2^j\big||x-y|-|t|\big|\big)^{-N}v_1(t,y)\md y,\]
		since \(t\geq T_0\gg1\) and \(|y|\leq1\), we have \(\big||x-y|-|t|\big|\geq\frac{1}{2}|t-|x||\), therefore by choosing \(N=\frac{n}{2}-\delta\) we conclude that
		\begin{equation*}
			\begin{split}
				|v_j|&\leq2^{\frac{n}{2}j}t^{-\frac{n}{2}}
				(2^j|t-|x||)^{-\frac{n}{2}+\delta}\|v_1(t,\cdot)\|_{L^1(\R^n)}\\
				&\lesssim2^{\delta j}(1+t)^{-\frac{n}{2}}
				(1+|t-|x||)^{-\frac{n}{2}+\delta}\|v_1(t,\cdot)\|_{L^1(\R^n)}.
			\end{split}
		\end{equation*}
		\paragraph{Case II.3 High frequency}
		The high frequency part can be handled similarly if one note that for
		\(t\geq T_0\gg1\) and \(|y|\leq1\),
		\[\big|t-2-|x-y|\big|\geq\frac{1}{2}\big|t-|x|\big|.\]
		
		\vspace{4mm}
		
		Collecting all the results above we have for \(\delta>0\) which can be arbitrarily small
		\begin{equation}\label{equ:sigma}
			|v(t,x)|\leq \ve C(v_0,v_1,T_0)(1+t)^{-\frac{n}{2}}(1+|t-|x||)^{-\frac{n}{2}+\delta}.
		\end{equation}
		Then we can compute the integral in the left hand side of \eqref{equ:homo-estimate} by \eqref{equ:sigma} and the polar coordinate transformation to get
		\begin{equation*}
			\begin{split}
				&\Big\|\big(t^2-|x|^2\big)^\gamma t^{\frac{1}{q}}v\Big\|_{L^q([2,\infty)\times\R^n)}^q \\
				\le& \big(\ve C(v_0,v_1,T_0)\big)^q\int_{2}^\infty\int_{|x|\leq t-1}
				\bigg\{\Big(t^2-|x|^2\Big)^\gamma(1+t)^{-\frac{n}{2}}
				(1+|t-|x||)^{-\frac{n}{2}+\delta}\bigg\}^q \md x\ t\md t
				\\
				\leq& \big(\ve C(v_0,v_1,T_0)\big)^q\int_{2}^\infty\int_0^{t-1}\Big\{\big(t+r\big)^\gamma
				\big(t-r\big)^\gamma(1+t)^{-\frac{n}{2}}
				(1+|t-r|)^{-\frac{n}{2}+\delta}\Big\}^q r^{n-1}\md r\ t\md t
				\\
				\lesssim &\big(\ve C(v_0,v_1,T_0)\big)^q\int_{2}^\infty t^{q(-\frac{n}{2}+\gamma)}
				\int_0^{t-1}\big(1+|r-t|\big)^{q(\gamma-\frac{n}{2}
					+\delta)}r^{n-1}\md r\ t\md t.
			\end{split}
		\end{equation*}
		Noting that \(\gamma<\frac{n}{2}-\frac{n+1}{q}\), thus
		\[q\left(\gamma-\frac{n}{2}+\delta\right)<q\left(\frac{n}{2}-\frac{n+1}{q}
		-\frac{n}{2}+\delta\right)=-(n+1)+n\delta.\]
		Hence if we choose \(\delta<\frac{1}{n}\), then
		\begin{equation}\label{equ:homo-compu}
			\Big\|\big(t^2-|x|^2\big)^\gamma t^{\frac{1}{q}}v\Big\|_{L^q([2,\infty)\times\R^n)}^q
			\leq\big(\ve C(v_0,v_1,T_0)\big)^q\int_{2}^\infty t^{q(-\frac{n}{2}+\gamma)+n}dt.
		\end{equation}
		Since 
		\[
		q\left(-\frac{n}{2}+\gamma\right)+n<q\left(-\frac{n}{2}+\frac{n}{2}-\frac{n+1}{q}\right)+n=-1,
		\]
		the integral in \eqref{equ:homo-compu} is convergent and Lemma \ref{lem:homogeneous} is proved.
		
	\end{proof}

	\subsection{Weighted Strichartz estimate for linear inhomogeneous equation}
	Next we turn to the weighted Strichartz estimate for linear inhomogeneous equation, our main result in this subsection is the following:
	\begin{theorem}\label{Str}
		Assume that \(n\geq4\) and \(\phi\) solves the Cauchy problem \eqref{equ:linear-inhomo}. Then for the function \(F\) with
		\begin{equation}\label{equ:supp-F}
			\textup{supp}F\subseteq\{(\tau,y)\mid \tau\geq2, |y|\leq \tau-1\},
		\end{equation}
		we have
		\begin{equation}\label{equ:supp-phi}
			\textup{supp}\phi\subseteq\{(t,x)\mid t\geq2, |x|\leq t-1\},
		\end{equation}
		and
		\begin{equation}\label{equ:inhomo-estimate}
			\begin{split}
				&\lp(t^2-|x|^2)^{n(\frac{1}{2}-\frac{1}{q})-\frac{1}{q}-\delta_1}
				t^{\frac{1}{q}}\phi(t,x)\rp_{L^q(\ra)}   \\
				\leq& C(\delta_1,\delta_2)\lp(\tau^2-|y|^2)^{\frac{1}{q}+\delta_2}\tau^{\frac{q-1}{q}}
				F(\tau,y)\rp_{L^{\frac{q}{q-1}}(\ra)},
			\end{split}
		\end{equation}
		where \(2<q\leq\frac{2(n+2)}{n}\).
	\end{theorem}
	
	In order to prove \eqref{equ:supp-phi}, recall that Theorem 1-Theorem 3 in \cite{P1} gives
	\[\phi(t,x)=2\int_{2}^{t}\int_0^{t-\tau}w[F](s,x;\tau)
	E(t,0;\tau,s;1,0)\md s\md \tau,\]
	where the kernel \(E\) can be represented by
	\[E(t,x;\tau,y;1,0)=\frac{1}{\sqrt{(t+\tau)^2-s^2}}F\lc\frac{1}{2},
	\frac{1}{2};1;\frac{(t-\tau)^2-s^2}{(t+\tau)^2-s^2}\rc,\]
	with \(F(\alpha,\beta;\gamma;z)\) denoting the Gauss hypergeometric function. Moreover, \(w[F](s,x;\tau)\) is the solution to the parameter-dependent Cauchy problem for the free wave equation
	\begin{equation*}
		\begin{cases}
			&\partial_s^2 w-\triangle w=0, \quad (s,x)\in(0,\infty)\times\R^n, \\
			&\phi(0,x)=F(\tau,x),\quad \partial_s\phi(0,x)=0.
		\end{cases}
	\end{equation*}
	By formulae (10) and (11) in \cite{P1} we see that if \((\tau,y)\in\text{supp}F\), then
	\begin{equation}\label{equ:supp-important}
		|x-y|\leq s\leq t-\tau.
	\end{equation}
	And hence \eqref{equ:supp-phi} follows by combining
	\eqref{equ:supp-important} and \eqref{equ:supp-F}.
	
	The remainder of this subsection will be devoted to the proof of \eqref{equ:inhomo-estimate}, which will be obtained by interpolating between two endpoints \(q=2\) and \(q=\frac{2(n+2)}{n}\). We first consider the case \(q=2\), where we need the following lemma:
	
	\begin{lemma}\label{lem3.3}
		Under the assumption of Theorem 3.2, we have for \(t\geq2\)
		\begin{equation}\label{equ:L2-estimate}
			\lp|t-|x||^{-\frac{1}{2}}t^{\frac{1}{2}}
			\phi(t,x)\rp_{L^2(\R^n)}
			\leq C\ln t
			\lp(\tau^2-|y|^2)^{\frac{1}{2}}\tau^{\frac{1}{2}}F(\tau,y)\rp
			_{L^2([2,t)\times\R^n)}.
		\end{equation}
	\end{lemma}
	The proof of Lemma~\ref{lem3.3} is similar to that of Lemma 3.1 in \cite{LZ0}, thus we omit the details. By \eqref{equ:L2-estimate}, we compute
	\begin{equation}\label{equ:L2-estimate-f}
		\begin{split}
			& \lp(t^2-|x|^2)^{-\frac{1}{2}-\delta_1}
			t^{\frac{1}{2}}\phi(t,x)\rp_{L^2(\ra)} \\
			\leq&\lc\int_{2}^{\infty}\Big|t^{-\frac{1}{2}-\delta_1}
			\lp|t-|x||^{-\frac{1}{2}}t^{\frac{1}{2}}\phi(t,x)\rp_{L^2(\R^n)}
			\Big|^2\md t\rc^{\frac{1}{2}} \\
			\lesssim&\lc\int_{2}^{\infty}\Big|t^{-\frac{1}{2}-\delta_1}
			\ln t\lp(\tau^2-|y|^2)^{\frac{1}{2}}
			\tau^{\frac{1}{2}}F(\tau,y)\rp_{L^2([2,t)\times\R^n)}\Big|^2
			\md t\rc^{\frac{1}{2}} \\
			\leq&\lc\int_{2}^{\infty}t^{-1-2\delta_1}(\ln t)^2\md t
			\rc^{\frac{1}{2}}\lp(\tau^2-|y|^2)^{\frac{1}{2}}
			\tau^{\frac{1}{2}}F(\tau,y)\rp_{L^2([2,\infty)\times\R^n)} \\
			\leq& C(\delta_1,\delta_2)\lp(\tau^2-|y|^2)^{\frac{1}{2}+\delta_2}
			\tau^{\frac{1}{2}+\delta_2}F(\tau,y)\rp_{L^2([2,\infty)\times\R^n)},
		\end{split}
	\end{equation}
	which yields \eqref{equ:inhomo-estimate} for \(q=2\). For the case \(q=\frac{2(n+2)}{n}\), we have the following lemma
	\begin{lemma}\label{Str1}
		Under the hypotheses of Theorem \ref{Str}, 
		\begin{equation*}
			\begin{split}
				&\lp(t^2-|x|^2)^{\frac{n}{2(n+2)}-\delta_1}
				t^{\frac{n}{2(n+2)}}\phi(t,x)
				\rp_{L^{\frac{2(n+2)}{n}}(\ra)} \\
				\leq &C(\delta_1,\delta_2)
				\lp(\tau^2-|y|^2)^{\frac{n}{2(n+2)}+\delta_2}
				\tau^{\frac{n+4}{2(n+2)}}F(\tau,y)\rp
				_{L^{\frac{2(n+2)}{n+4}}(\ra)}.
			\end{split}
		\end{equation*}
	\end{lemma}

	In order to derive Lemma \ref{Str1}, we follow the argument in Georgiev, Lindblad and Sogge \cite{Gls1}. Let us see that we can estimate \(\phi\) if the norm is taken over a set where \(T/2\leq t\leq T\) and \(F(t, x)\) vanishes when 
	\(|x|\geq t-1\). To be more specific, if we let \(\phi=\phi^1+\phi^0\), where \(\Box\phi^j=F^j, j=0,1\) with zero data and if \(F^1(t,x)=F(t,x)\), for \(t\geq T/10\), but zero otherwise then we claim that
	\begin{equation}\label{equ:Lq-estimate-0}
		\begin{split}
			&\lp(t^2-|x|^2)^{\frac{n}{2(n+2)}-\delta_1}
			t^{\frac{n}{2(n+2)}}\phi^0\rp_{L^{\frac{2(n+2)}{n}}\left(\{(t,x)\mid
				\frac{T}{2}\leq t\leq T\}\right)} \\
			\leq& C(\delta_1,\delta_2)T^{-\delta_1}\ln T
			\lp(t^2-|x|^2)^{\frac{n}{2(n+2)}+\delta_2}t^{\frac{n+4}{2(n+2)}}F^0\rp
			_{L^{\frac{2(n+2)}{n+4}}(\ra)},
		\end{split}
	\end{equation}
	and
	\begin{equation}\label{equ:Lq-estimate-1}
		\begin{split}
			&\lp(t^2-|x|^2)^{\frac{n}{2(n+2)}-\delta_1}
			t^{\frac{n}{2(n+2)}}\phi^1\rp_{L^{\frac{2(n+2)}{n}}\left(\{(t,x)\mid
				\frac{T}{2}\leq t\leq T\}\right)} \\
			\leq &C(\delta_1,\delta_2)T^{-\frac{\delta_1}{2}}
			\lp(t^2-|x|^2)^{\frac{n}{2(n+2)}+\delta_2}t^{\frac{n+4}{2(n+2)}}F^1\rp
			_{L^{\frac{2(n+2)}{n+4}}(\ra)}.
		\end{split}
	\end{equation}
	Note that \(\phi^j\) and \(F^j\), like \(\phi\) and \(F\), vanish when \(t-|x|\leq1\).
	
	We first prove claim \eqref{equ:Lq-estimate-0}. Note that
	\(\Box\phi^0=F^0\) and \(F^0(t,x)=F(t,x)\) for \(t\leq T/10\) and \(0\) otherwise. Then for \(t\geq T/2\) and \(\tau\leq T/10\) we have \(t-\tau\geq T/4\). By the \(L^1-L^\infty\) estimate \eqref{equ:cutoff-infty-1-1} in Lemma \ref{A2} below we have 
	\begin{equation*}
		\|\phi^0(t,\cdot)\|
		_{\dot{B}^{-\frac{n}{2}}_{\infty,\infty}(\R^n)}\leq C(\ln t)^2\int_{2}^{t}|t-\tau|^{-\frac{n}{2}}
		\tau\|F^0(\tau,\cdot)\|_{L^1(\R^n)}\ \md\tau,
	\end{equation*}
	which implies that
	\begin{equation}\label{equ:infty-large-time}
		\|\phi^0(t,x)\|
		_{L^\infty\left([\frac{T}{2},T],\dot{B}^{-\frac{n}{2}}
			_{\infty,\infty}(\R^n)\right)}\lesssim T^{-\frac{n}{2}}(\ln T)^2
		\left\|\tau F(\tau,y)\right\|_{L^1\left([2,\frac{T}{10}]\times\R^n\right)}.
	\end{equation}
	On the other hand, by \eqref{equ:Morawetz} with $\mu=1$ we have for \(\phi^0\)
	\begin{equation}\label{equ:2-large-time}
		\begin{split}
			&\sup_{2\leq t\leq T}   t^{\frac{1}{2}}\lp(1+|u|)^{\frac{1}{2}}
			\nabla_{x}\phi(t,\cdot)\rp_{L^2(\R^n)}
			\lesssim
			\left\|(1+|u|)^{\frac{1}{2}+\delta}(1+|\underline{u}|)^{\frac{1}{2}+\delta}
			\tau^{\frac{1}{2}}F(\tau,y)\right\|_{L^2([2,T]\times\R^n)}.
		\end{split}
	\end{equation}
	Interpolating between \eqref{equ:infty-large-time} and
	\eqref{equ:2-large-time} yields that
	\begin{equation*}
		\begin{split}
			&\left\|(1+|u|)^{\frac{n}{2(n+2)}}t^{\frac{n}{2(n+2)}
			}\phi^0(t,x)\right\|
			_{L^\infty\left([\frac{T}{2},T],L^{\frac{2(n+2)}{n}}
				(\R^n)\right)} \\
			\lesssim& T^{-\frac{n}{n+2}}(\ln T)^{\frac{4}{n+2}}
			\left\|(1+|u|)^{\frac{n+2n\delta}{2(n+2)}}
			(1+|\underline{u}|)^{\frac{n+2n\delta}{2(n+2)}}\tau^{\frac{n+4}{2(n+2)}} F(\tau,y)\right\|_{L^{\frac{2(n+2)}{n+4}}\left([2,\frac{T}{10}]\times\R^n\right)}.
		\end{split}
	\end{equation*}
	Multiplying the both sides of the above inequality with \(t^{\frac{n}{2(n+2)}-\delta_1}\), and then taking \(L^{\frac{2(n+2)}{n}}\) norm with respect to \(t\), we get
	\begin{equation*}
		\begin{split}
			&\lp(t^2-|x|^2)^{\frac{n}{2(n+2)}-\delta_1}
			t^{\frac{n}{2(n+2)}}\phi^0\rp_{L^{\frac{2(n+2)}{n}}\left(\{(t,x)\mid
				\frac{T}{2}\leq t\leq T\}\right)} \\
			\leq &C(\delta_1,\delta_2)T^{-\delta_1}\ln T
			\lp(t^2-|x|^2)^{\frac{n}{2(n+2)}+\delta_2}
			t^{\frac{n+4}{2(n+2)}}F^0\rp
			_{L^{\frac{2(n+2)}{n+4}}([2,T]\times\R^n)},
		\end{split}
	\end{equation*}
	which implies the claim \eqref{equ:Lq-estimate-0}.
	
	In the next step we establish \eqref{equ:Lq-estimate-1}, the point of the proof for \eqref{equ:Lq-estimate-1} is the following proposition.
	\begin{proposition}\label{prop3.1}
		Under the hypotheses of Theorem 3.2, if moreover \(T\geq4\),
		\begin{equation}\label{equ:Lq-estimate}
			\begin{split}
				&\lp(t^2-|x|^2)^{\frac{n}{2(n+2)}-\delta_1}
				t^{\frac{n}{2(n+2)}}\phi(t,x)\rp_{L^{\frac{2(n+2)}{n}}\left(\{(t,x)\mid
					\frac{T}{2}\leq t\leq T, |x|\leq\frac{3t}{4}\}\right)} \\
				\leq& C(\delta_1,\delta_2)T^{-\delta_1}
				\lp(t^2-|x|^2)^{\frac{n}{2(n+2)}+\delta_2}t^{\frac{n}{2(n+2)}}
				F(\tau,y)\rp_{L^{\frac{2(n+2)}{n+4}}\left([2,T]\times\R^n\right)},
			\end{split}
		\end{equation}
	\end{proposition}
	\begin{proof}[Proof of Proposition~\ref{prop3.1}]
		Define a smooth function \(\chi\left(\frac{|x|}{t}\right)\) with
		\begin{equation*}
			\chi\equiv\left\{ \enspace
			\begin{aligned}
				&0, && \text{if}\ \ \frac{|x|}{t}\geq\frac{7}{8}, \\
				&1, && \text{if}\ \ \frac{|x|}{t}\leq\frac{3}{4}.
			\end{aligned}
			\right.
		\end{equation*}
		First by \eqref{equ:Morawetz} with $\mu=1$ in Theorem \ref{thm:Morawetz} we have,
		\begin{equation}\label{equ:L2-weight}
			\begin{split}
				t\|(\chi\phi)(t,\cdot)\|_{\dot{H}^1(\R^n)}=&t\|\nabla(\chi\phi)(t,\cdot)\|_{L^2(\R^n)} \\
				\leq& C\left(t^{\frac{1}{2}}\lp\frac{(t-|x|)^{\frac{1}{2}}\phi}{|x|}\rp
				_{L^2(|x|\leq\frac{7t}{8})}+t^{\frac12}\left\|(t-|x|)^{\frac{1}{2}}\nabla\phi\right\| _{L^2(|x|\leq\frac{7t}{8})} \right) \\
				\leq& C \|(\tau+|y|)^{\frac{1}{2}+\delta}(\tau-|y|)^{\frac{1}{2}+\delta}
				\tau^{\frac{1}{2}}F(\tau,y)\|
				_{L^2([2,t)\times\R^n)}.
			\end{split}
		\end{equation}
		Second we can assert that
		\begin{equation}\label{equ:L-infty-weight}
			\|(\chi\phi)(t,\cdot)\|_{\dot{B}^{-\frac{n}{2}}_{\infty,\infty}(\R^n)}
			\leq C(\ln t)^2\int_{2}^{t}|t-\tau|^{-\frac{n}{2}}
			\tau\|F(\tau,\cdot)\|_{L^1(\R^n)}\md\tau.
		\end{equation}
		In order to prove \eqref{equ:L-infty-weight}, we will use the Littlewood-Paley decomposition. By the notation in \eqref{equ:cut-off}, setting \(\phi_j=\rho_j(D)\phi\) and \(F_j=\rho_j(D)F\), then
		\[\phi=\sum_{j=-\infty}^{\infty}\phi_j, \qquad F=\sum_{j=-\infty}^{\infty}F_j.\]
		Notice that
		\begin{equation*}
			\|(\chi\phi)(t,\cdot)\|
			_{\dot{B}^{-\frac{n}{2}}_{\infty,\infty}(\R^n)}
			=\sup_{j\in\mathbb{Z}}2^{-\frac{n}{2}j}\|(\chi\phi)_j(t,\cdot)\|_{L^\infty(\R^n)},
		\end{equation*}
		and we compute the main term on the right side as
		\begin{equation}\label{equ:decomp-estimate}
			\begin{split}
				\|(\chi\phi)_j\|_{L^\infty(\R^n)}
				&\leq\sum_{k=-\infty}^{\infty}\|(\chi\phi_k)_j\|_{L^\infty(\R^n)}\leq\sum_{k=-\infty}^{j+5}\|\phi_k\|_{L^\infty(\R^n)}
				+\sum_{k=j+5}^{\infty}\|(\chi\phi_k)_j\|_{L^\infty(\R^n)},
			\end{split}
		\end{equation}
		and we have by \eqref{equ:cutoff-infty-1-1} in Lemma \ref{A2} below
		\begin{equation*}
			\|\phi_k(t,\cdot)\|_{L^\infty(\R^n)}\leq C2^{\frac{n}{2}k}(\ln t)^2\int_{2}^{t}|t-\tau|^{-\frac{n}{2}}
			\tau\|F_k(\tau,\cdot)\|_{L^1(\R^n)}\ \md\tau.
		\end{equation*}
		Thus
		\begin{equation}\label{equ:k-small}
			\begin{split}
				\sum_{k\leq j+5}    \|\phi_k(t,\cdot)\|_{L^\infty(\R^n)}\lesssim&\sum_{k\leq j+5}2^{\frac{n}{2}k}(\ln t)^2\int_{2}^{t}
				|t-\tau|^{-\frac{n}{2}}\tau
				\|F_k(\tau,\cdot)\|_{L^1(\R^n)}\ \md\tau \\
				\lesssim &2^{\frac{n}{2}j}(\ln t)^2\int_{2}^{t}
				|t-\tau|^{-\frac{n}{2}}\tau
				\|F(\tau,\cdot)\|_{L^1(\R^n)}\ \md\tau.
			\end{split}
		\end{equation}
		For the second term on the right hand side of \eqref{equ:decomp-estimate}, by Bernstein inequality we obtain
		\begin{equation}\label{equ:Bernstein-1}
			\begin{split}
				\sum_{k>j+5}\|(\chi\phi_k)_j\|_{L^\infty(\R^n)} & \lesssim 2^{\frac{n}{2}j}\sum_{k>j+5}
				\|(\chi\phi_k)_j\|_{L^2(\R^n)}. \\
			\end{split}
		\end{equation}
		Then we turn to handle \(\|(\chi\phi_k)_j\|_{L^2(\R^n)}\) for \(k>j+5\). If \(2^k\leq t^{-1}\), then by \eqref{equ:cutoff-2-1-low} in Lemma \ref{A2} below and the Bernstein inequality
		\begin{equation}\label{equ:L2-1}
			\begin{split}
				\|(\chi\phi_k)_j(t,\cdot)\|_{L^2(\R^n)} & \leq \|\phi_k(t,\cdot)\|_{L^2(\R^n)}\\
				& \lesssim  t^{-1}(\ln t)^2\int_{2}^{t}\tau\lp
				F_k(\tau,\cdot)\rp_{\dot{H}^{-1}(\R^n)}\md\tau \\
				& \lesssim t^{-1}(\ln t)^2\int_{2}^{t}\tau 2^{\frac{n-2}{2}k}
				\lp F_k(\tau,\cdot)\rp_{L^1(\R^n)}\md\tau.
			\end{split}
		\end{equation}
		On the other hand, if \(2^k>t^{-1}\), then in a similar way we have by \eqref{equ:cutoff-2-1-high} in Lemma \ref{A2} below
		\begin{equation}\label{equ:L2-2}
			\begin{split}
				\|(\chi\phi_k)_j(t,\cdot)\|_{L^2(\R^n)}\leq&\sum_{|l-k|<10}\|\chi_l\|_{L^\infty}
				\|\phi_k(t,\cdot)\|_{L^2(\R^n)} \\
				\lesssim& t^{-N}2^{-kN}\|D^N\chi_l\|_{L^\infty}
				t^{-\frac{1}{2}}2^{\frac{1}{2}k}\ln t\int_2^t\tau\lp F_k(\tau,\cdot)\rp_{\dot{H}^{-1}(\R^n)}\md\tau\\
				\lesssim &t^{-N}2^{-kN}t^{-\frac{1}{2}}2^{\frac{1}{2}k}
				2^{\frac{n-2}{2}k}\ln t\int_2^t\tau\lp F_k(\tau,\cdot)\rp_{L^1(\R^n)}\md\tau. \\
			\end{split}
		\end{equation}
		Combining \eqref{equ:Bernstein-1}-\eqref{equ:L2-2}, we arrive at
		\begin{equation}\label{equ:k-large}
			\begin{split}
				\sum_{k>j+5}\|(\chi\phi_k)_j\|_{L^\infty(\R^n)}\lesssim&2^{\frac{n}{2}j}\sum_{2^k\leq t^{-1}}t^{-1}(\ln t)^2\int_{2}^{t}\tau 2^{\frac{n-2}{2}k}
				\lp F_k(\tau,\cdot)\rp_{L^1(\R^n)}\md\tau \\
				& +2^{\frac{n}{2}j}\sum_{2^k>t^{-1}}t^{-\frac{1}{2}}t^{-N}
				2^{(\frac{n-1}{2}-N)k}(\ln t)^2\int_2^t\tau\lp F_k(\tau,\cdot)\rp_{L^1(\R^n)}\md\tau \\
				\lesssim &2^{\frac{n}{2}j}t^{-\frac{n}{2}}(\ln t)^2\int_2^t\tau\lp F(\tau,\cdot)\rp_{L^1(\R^n)}\md\tau \\
				\lesssim &2^{\frac{n}{2}j}(\ln t)^2\int_2^t|t-\tau|^{-\frac{n}{2}}\tau
				\lp F(\tau,\cdot)\rp_{L^1(\R^n)}\md\tau.
			\end{split}
		\end{equation}
		Collecting \eqref{equ:k-small} and \eqref{equ:k-large}, the claim in \eqref{equ:L-infty-weight} is verified. After that, interpolating between \eqref{equ:L2-weight} and \eqref{equ:L-infty-weight} we obtain
		\begin{equation}\label{equ:interpolation}
			\begin{split}
				&\lp (t-|x|)^{\frac{n}{2(n+2)}}
				t^{\frac{n}{2(n+2)}}\phi\rp_{L^{\frac{2(n+2)}{n}}
					(\{x\mid|x|\leq\frac{3t}{4}\})}
				\leq\lp t^{\frac{n}{n+2}}
				\chi\phi\rp_{L^{\frac{2(n+2)}{n}}(\R^n)} \\
				\leq &C
				\lp|t-\tau|^{-\frac{n}{n+2}}(\tau^2-|y|^2)^{\frac{n}{2(n+2)}+\frac{2}{n+2}\delta}
				\tau^{\frac{n+4}{2(n+2)}}F\rp_{L^{\frac{2(n+2)}{n+4}}(\rb)}.
			\end{split}
		\end{equation}
		
		For the remaining part of the proof, we show how to establish \eqref{equ:Lq-estimate} by applying \eqref{equ:interpolation}. First noting that \eqref{equ:supp-important} implies
		\[t-\tau\geq|x-y|\geq|x|-|y|,\]
		thus if \(\tau-|y|\geq T/16\), and \(\tau\geq T/4\), then the weights on both sides of \eqref{equ:Lq-estimate} are comparable and hence can be removed and hence it suffices to prove
		\begin{equation}\label{equ:Lq-estimate-2}
			\begin{split}
				&\lp
				t^{\frac{n}{2(n+2)}}\phi(t,x)\rp_{L^{\frac{2(n+2)}{n}}\left(\{(t,x)\mid
					\frac{T}{2}\leq t\leq T, |x|\leq\frac{3t}{4}\}\right)}
				\leq C(\delta_1,\delta_2)T^{\delta_1+2\delta_2}
				\lp t^{\frac{n+4}{2(n+2)}}
				F(\tau,y)\rp_{L^{\frac{2(n+2)}{n+4}}([2,T]\times\R^n)},
			\end{split}
		\end{equation}
		but \eqref{equ:Lq-estimate-2} will follow from the Strichartz estimate without characteristic weight, see Lemma~\ref{lem:time-weighted-Strichartz}.
		If \(\tau\leq T/4\), then
		\[t-\tau\geq \frac{T}{2}-\frac{T}{4}\geq\frac{T}{4},\]
		hence \eqref{equ:interpolation} gives
		\begin{align*}\label{equ:T}
			& \lp (t-|x|)^{\frac{n}{2(n+2)}}
			t^{\frac{n}{2(n+2)}}\phi\rp_{L^{\frac{2(n+2)}{n}}
				\left(\{x\mid|x|\leq\frac{3t}{4}\}\right)}
			\lesssim T^{-\frac{n}{n+2}}\lp(\tau^2-|y|^2)^{\frac{n}{2(n+2)}+\frac{2}{n+2}\delta}
			\tau^{\frac{n+4}{2(n+2)}}F\rp_{L^{\frac{2(n+2)}{n+4}}(\rb)}.
		\end{align*}
		Then by a computation similar to \eqref{equ:L2-estimate-f}, \eqref{equ:Lq-estimate} follows from \eqref{equ:interpolation} if we choose \(\delta\) small enough. Thus we are left with the case \(\tau\geq T/4, \tau-|y|\leq T/16\). In this case, noting that \(t-\tau\geq|x-y|\geq|x|-|y|\), then by (3.29) in \cite{LZ} we see that \(t-\tau\geq T/32\), hence \eqref{equ:Lq-estimate} follows again from \eqref{equ:interpolation} if we choose \(\delta\) small enough.
		
	\end{proof}

	Let us see how can we derive claim \eqref{equ:Lq-estimate-1} from Proposition~\ref{prop3.1}. Specifically, we note that \eqref{equ:Lq-estimate-1} follows from the further localized bounds
	\begin{equation}\label{equ:Lq-estimate-k}
		\begin{split}
			&\lp(t^2-|x|^2)^{\frac{n}{2(n+2)}-\frac{\delta_1}{2}}
			t^{\frac{n}{2(n+2)}}\phi^1\rp_{L^{\frac{2(n+2)}{n}}
				\left(\{(t,x)\mid\frac{T}{2}\leq t\leq T,2^{k-1}\leq t-|x|\leq2^k\}\right)} \\
			\leq& C(\delta_1,\delta_2)(2^kT)^{-\delta_1}
			\lp(t^2-|x|^2)^{\frac{n}{2(n+2)}+\delta_2}
			t^{\frac{n+4}{2(n+2)}}F^1\rp_{L^{\frac{2(n+2)}{n+4}}(\ra)}.
		\end{split}
	\end{equation}
	Clearly in what follows we may assume that \(2^k\leq 2T\), since
	otherwise the condition in the left will not be satisfied. Also, if we set \(T_k=T/2^k\) and let \(\phi^1_k(t,x)=\phi^1(2^kt,2^kx)\) and \(F_k^1(t,x)=2^{2k}F^1(2^kt,2^kx)\), then \eqref{equ:Lq-estimate-k} can be derived from
	\begin{equation}\label{equ:Lq-estimate-k-1}
		\begin{split}
			&\lp(t^2-|x|^2)^{\frac{n}{2(n+2)}-\delta_1}
			t^{\frac{n}{2(n+2)}}\phi^1_k\rp_{L^{\frac{2(n+2)}{n}}
				\left(\{(t,x)\mid\frac{T_k}{2}\leq t\leq T_k,\frac{1}{2}\leq t-|x|\leq1\}\right)} \\
			\leq& C(\delta_1,\delta_2)T_k^{-\delta_1}
			\lp(t^2-|x|^2)^{\frac{n}{2(n+2)}+\delta_2}
			t^{\frac{n+4}{2(n+2)}}F^1_k\rp_{L^{\frac{2(n+2)}{n+4}}([2^{1-k},\infty)\times\R^n)}.
		\end{split}
	\end{equation}
	
	To use all of this we shall need the following two lemmas, which are essentially from Georgiev-Lindblad-Sogge \cite{Gls1}.
	\begin{lemma}\label{G1}
		Let \(E_+\) be the forward fundamental solution for \(\Box+\partial_t/t\). If \(0\leq t-|x|\leq1\), \(t/10\leq \tau\leq t\), and \(s-1\leq|y|\leq \tau\), then
		\[\left|\frac{x}{|x|}-\frac{y}{|y|}\right|\leq\frac{C}{t},
		\quad \text{if} \quad (t,x,s,y)\in\textup{supp}E_+(t-s,x-y),\]
		for some uniform constant C.
	\end{lemma}
	
	\begin{lemma}\label{G2}
		Suppose that \(K(x,y)\) is a measurable function on \(\R^m\times\R^n\) and set
		\[Tf(x)=\int K(x,y)f(y)\md y.\]
		Suppose further that we can write \(\R^m\) and \(\R^n\) as disjoint unions \(\R^m=\cup_{j\in\mathbb{Z}^d}A_j\) and
		\(\R^n=\cup_{k\in\mathbb{Z}^k}A_j\), where if \(x\in A_j\), then \(K(x,y)=0\) when \(y\in B_k\) with \(|j-k|>C\), for some uniform constant \(C\). Then, if we let \(T_{jk}\) denote the integral with kernel \(K_{jk}\), where \(K_{jk}(x,y)=K(x,y)\) if \((x,y)\in A_j\times B_k\) and zero otherwise,
		\[\|T\|_{L^p\rightarrow L^p}\leq(2C+1)^d
		\cdot\sup_{j,k}\|T\|_{L^p\rightarrow L^p},\]
		provided that \(1\leq p\leq q\leq\infty\).
	\end{lemma}
	Noting that \eqref{equ:supp-important} implies
	\[t-\tau\geq|x-y|\geq|x|-|y|,\]
	then the proof of Lemma \ref{G1} is similar to that of Lemma 2.2 in \cite{Gls1}, while Lemma \ref{G2} is exactly Lemma 2.3 in \cite{Gls1}, thus we omit the details.
	
	With the two lemmas above we are ready to obtain \eqref{equ:Lq-estimate-k} from \eqref{equ:Lq-estimate}. We first notice that it is enough to prove the variant of \eqref{equ:Lq-estimate-k} where in the left we also assume that \(\big|\frac{x}{|x|}-\nu\big|\leq\frac{C}{\sqrt{T_k}}\) for some \(\nu\in\mathbb{S}^{n-1}\). Next, we let \(\omega=\frac{(t,x)}{\sqrt{t^2-|x|^2}}\) denote the projection of \((t,x)\) onto the unit hyperboloid \(\mathbb{H}^n\). We realize that if \((t_j,x_j)\), \(j=1,2\) are two points in the set with
	\begin{equation}\label{equ:small set}
		\frac{T_k}{2}\leq t\leq T_k, \quad \frac{1}{2}\leq t-|x|\leq1, \quad \left|\frac{x}{|x|}-\nu\right|\leq\frac{C}{\sqrt{T_k}},
	\end{equation}
	then we must have \(\textup{dist}(\omega_1,\omega_2)\leq C_0\), for some uniform constant with``dist" denoting the distance on \(\mathbb{H}^n\) with respect to the restriction of the Lorentz metric \(dx^2-dt^2\) to the hyperboloid. Hence, we can introduce transforms similar to Lorentz rotation, and such transforms send this set to the``middle" of the light cone. More specifically, one note that the equation
	\begin{equation*}
		\partial_t^2 \phi-\triangle \phi+\frac{1}{t}\partial_t\phi=F(t,x), \quad (t,x)\in[2,\infty)\times\R^n
	\end{equation*}
	is equivalent to the ultra-hyperbolic equation in \(\R^2\times\R^n\) with radial symmetric assumption in \(\R^2\):
	\begin{equation}\label{equ:ultrahyperbolic}
		\partial_{t_1}^2\phi+\partial_{t_2}^2\phi-\Delta_x\phi=F(t_1,t_2,x),
	\end{equation}
	where \(\phi(t_1,t_2,x)=\phi(t,x)\), \(F(t_1,t_2,x)=F(t,x)\) with \(|t|=\sqrt{t_1^2+t_2^2}\). Then \eqref{equ:Lq-estimate} is equivalent to the following inequalities in \(\R^2\times\R^n\)
	\begin{equation}\label{equ:Lq-estimate-ultrahyper}
		\begin{split}
			&\lp(t_1^2+t_2^2-|x|^2)^{\frac{n}{2(n+2)}-\delta_1}
			\phi(t_1,t_2,x)\rp_{L^{\frac{2(n+2)}{n}}\left(\{(t_1,t_2,x)\mid
				\frac{T}{2}\leq \sqrt{t_1^2+t_2^2}\leq T, |x|\leq\frac{3\sqrt{t_1^2+t_2^2}}{4}\}\right)} \\
			\leq &C(\delta_1,\delta_2)T^{-\delta_1}
			\lp(\tau_1^2+\tau_2^2-|x|^2)^{\frac{n}{2(n+2)}+\delta_2}
			F(\tau_1,\tau_2,y)\rp_{L^{\frac{2(n+2)}{n+4}}
				\left(\{(\tau_1,\tau_2,y)\mid2\leq\sqrt{\tau_1^2+\tau_2^2}\leq T, |y|\in\R^n\}\right)}.
		\end{split}
	\end{equation}
	Introducing the transforms \(L_1^j, L_2^j, j=1,2,\cdots, n\) as follows
	\begin{equation*}
		L_1^j:\left\{ \enspace
		\begin{aligned}
			&t_1, \\
			&t_2, \\
			&x_1, \\
			&\cdots \\
			&x_j, \\
			&\cdots \\
			&x_n,
		\end{aligned}
		\right.
		\longrightarrow
		\left\{ \enspace
		\begin{aligned}
			&\tilde{t}_1=\gamma(t_1-vx_j), \\
			&\tilde{t}_2=t_2, \\
			&\tilde{x}_1=x_1, \\
			&\cdots \\
			&\tilde{x}_j=\gamma(x_j-vt_1), \\
			&\cdots \\
			&\tilde{x}_n=x_n,
		\end{aligned}
		\right.
		\quad
		L_2^j:\left\{ \enspace
		\begin{aligned}
			&t_1, \\
			&t_2, \\
			&x_1, \\
			&\cdots \\
			&x_j, \\
			&\cdots \\
			&x_n,
		\end{aligned}
		\right.
		\longrightarrow
		\left\{ \enspace
		\begin{aligned}
			&\tilde{t}_1=t_1, \\
			&\tilde{t}_2=\gamma(t_2-vx_j), \\
			&\tilde{x}_1=x_1, \\
			&\cdots \\
			&\tilde{x}_j=\gamma(x_j-vt_2), \\
			&\cdots \\
			&\tilde{x}_n=x_n,
		\end{aligned}
		\right.
	\end{equation*}
	where \(v\in(0,1)\) is a suitable constant and \(\gamma=\frac{1}{\sqrt{1-v^2}}\).
	Direct computation shows that \(L_1^j, L_2^j, j=1,2,\cdots, n\) keeps the equation
	\eqref{equ:ultrahyperbolic} and the weight \(t_1^2+t_2^2-|x|^2\) invariant, and \(L_1^j, L_2^j, j=1,2,\cdots, n\)  would send the set in \eqref{equ:small set} into the new one
	\[
	\left\{(t_1,t_2,x)\Big| |x|\leq\frac{3\sqrt{t_1^2+t_2^2}}{4}\right\}.
	\]
	This observation together with \eqref{equ:Lq-estimate-ultrahyper} yield
	\begin{equation}\label{equ:Lq-estimate-ultrahyper-k}
		\begin{split}
			&\lp(t_1^2+t_2^2-|x|^2)^{\frac{n}{2(n+2)}-\delta_1}
			\phi^1_k\rp_{L^{\frac{2(n+2)}{n}}
				\left(\{(t_1,t_2,x)\mid\frac{T_k}{2}\leq \sqrt{t_1^2+t_2^2}\leq T_k,\frac{1}{2}\leq \sqrt{t_1^2+t_2^2}-|x|\leq1\}\right)} \\
			\leq &C(\delta_1,\delta_2)
			\lp(\tau_1^2+\tau_2^2-|y|^2)^{\frac{n}{2(n+2)}+\delta_2}
			F^1_k\rp_{L^{\frac{2(n+2)}{n+4}}
				\left(\{(\tau_1,\tau_2,y)\mid2^{1-k}\leq\sqrt{\tau_1^2+\tau_2^2}\leq T_k, |y|\in\R^n\}\right)},
		\end{split}
	\end{equation}
	and \eqref{equ:Lq-estimate-ultrahyper-k} is equivalent to \eqref{equ:Lq-estimate-k-1}. Finally, by taking \(T=2^j,j=2,3,\cdots,k\) in Proposition~\ref{prop3.1} and summing up, Lemma \ref{Str1} follows.
	\appendix
	
	\section*{Acknowledgement}
	
	The authors would like to express the sincere thank to Professor Huicheng Yin and Professor Yi Zhou for many helpful discussions.
	
	Daoyin He was supported by the Jiangsu Provincial Scientific Research Center of Applied Mathematics (No. BK20233002 and 2242023R40009). Ning-An Lai was partially supported by NSFC (No. 12271487, W2521007 and 12171097). 
	
	\section{Appendix}
	To begin with, we give the energy estimates which are necessary for the proof of Theorem~\ref{thm:Morawetz}.
	\begin{lemma}\label{lem:energy}
		Consider the solution of \eqref{equ:linear-inhomo} with \(n\geq2\) and \(t_0\geq0\). Then for \(\gamma>1\) and \(\mu\in(0,2)\) we have the following energy estimates
		\begin{equation}\label{equ:energy}
			\underset{t_0< t\leq T}{\textup{sup}} t^{\frac{\mu}{2}} \left[  \left\|  \nabla_{x} \phi \right\| _{L^2 (\mathbb R^n)}\right]\leq C
			\left \| \tau ^{\frac{\mu}{2}} \left(  1+ |u|  \right)^{\frac{\gamma}{2}} \left(  1+ \underline{u}  \right)^{\frac{\gamma}{2}} F \right\|  _{L^2 ([t_0,T] \times \mathbb R^n  )},
		\end{equation}
		where the positive constant \(C\) depends on \(n\), \(\mu\) and \(\gamma\).
	\end{lemma}
	\begin{proof}
		By \cite[(56)]{Wirth-0}, we have
		\begin{equation*}
			\hat{\phi}(t,\xi)=\int_{2}^{t}\Phi_1(t,\tau,\xi)\hat{F}(\tau,\xi)\md \xi,
		\end{equation*}
		where
		\[\Phi_1(t,\tau,\xi)=-\frac{i\pi}{4}\frac{t^\nu}{\tau^{\nu-1}}
		\left[H_\nu^-(\tau|\xi|)H_\nu^+(t|\xi|)
		-H_\nu^+(\tau|\xi|)H_\nu^-(t|\xi|)\right],\]
		where $H_{\nu}^{\pm}$ are defined as in \eqref{equ:z-large-1}.
		Since the property of \(H_\nu^+\) and \(H_\nu^-\) are similar, by neglecting the constants one may write
		\begin{equation}\label{equ:phi-1}
			\phi(t,x)=\int_{t_0}^{t}\int_{\R^n}\int_{\R^n}e^{i(x-y)\cdot\xi}
			\tau H_\nu^-\left(\tau|\xi|)H_\nu^+(t|\xi|\right)F(\tau,y)\md y\md\xi\md\tau.
		\end{equation}
		
		We estimate \(\|  \nabla_x \phi \| _{L^2 (\mathbb R^n)}\) by dyadic decomposition. Note that \eqref{equ:phi-1} implies
		\begin{equation*}
			\phi(t,x)=\int_{t_0}^{t}\int_{\R^n}\int_{\R^n}e^{i(x-y)\cdot\xi}
			\tau H_\nu^-(\tau|\xi|)H_\nu^+(t|\xi|)F(\tau,y)\md y\md\xi\md\tau.
		\end{equation*}
		Hence for \(\phi_k=\rho_k(D)\phi\) and \(F_k=\rho_k(D)F\)
		\begin{equation}\label{equ:phi-2}
			\phi_k(t,x)=\int_{t_0}^{t}\int_{\R^n}\int_{\R^n}e^{i(x-y)\cdot\xi}
			\tau H_\nu^-(\tau|\xi|)H_\nu^+(t|\xi|)F_k(\tau,y)\md y\md\xi\md\tau.
		\end{equation}
		Next we estimate \(\phi_k\) according to the value of \(\mu\). 
		
		\noindent\textbf{Case i \(\mathbf{\mu=1}\)} We have \(\nu=\frac{1-\mu}{2}=0\). By \eqref{equ:z-small-1}, if \(2^k\leq t^{-1}\), then
		\begin{equation}\label{equ:low-2-d}
			\begin{split}
				\lp \phi_k(t,\cdot)\rp_{L^{2}(\R^n)}\leq&
				\int_{t_0}^{t}\tau\lp\big|\ln(\tau|\xi|)\ln(t|\xi|)\big|
				\hat{F}_k(\tau,\xi)\rp_{L^{2}(\R^n)}\md\tau \\
				\leq &t^{-\frac{1}{2}}\int_{t_0}^{t}\tau^{\frac{1}{2}}\lp\big|\ln(\tau|\xi|)(\tau|\xi|)^{\frac{1}{2}}
				\ln(t|\xi|)(t|\xi|)^{\frac{1}{2}}\big|
				\hat{F}_k(\tau,\xi)\rp_{\dot{H}^{-1}(\R^n)}\md\tau \\
				\leq& t^{-\frac{1}{2}}
				\int_{t_0}^{t}\tau^{\frac{1}{2}}\lp
				F_k(\tau,\cdot)\rp_{\dot{H}^{-1}(\R^n)}\md\tau.
			\end{split}
		\end{equation}
		If \(2^k>t^{-1}\), with out loss of generality we make an extension of \(F\) and assume that \(F(\tau,y)\equiv0\) if \(\tau<t_0\) then
		\begin{equation}\label{equ:high-2-d}
			\begin{split}
				\lp\phi_k(t,\cdot)\rp_{L^2(\R^n)}\lesssim& \int_{0}^{2^{-k}}\tau\lp\ln(\tau|\xi|)
				(t|\xi|)^{-\frac{1}{2}}|\hat{F}_k|\rp_{L^2(\R^n)}\md\tau +\int_{2^{-k}}^t\tau\lp(\tau|\xi|)^{-\frac{1}{2}}
				(t|\xi|)^{-\frac{1}{2}}|\hat{F}_k|\rp_{L^2(\R^n)}\md\tau \\
				\lesssim &t^{-\frac{1}{2}}
				\int_{0}^{2^{-k}}\tau^{\frac{1}{2}}
				\lp\ln(\tau|\xi|)(\tau|\xi|)^{\frac{1}{2}} F_k(\tau,\cdot)\rp_{\dot{H}^{-1}(\R^n)}\md\tau
				+t^{-\frac{1}{2}}\int_{2^{-k}}^t\tau^{\frac{1}{2}}\lp F_k(\tau,\cdot)\rp_{\dot{H}^{-1}(\R^n)}\md\tau \\
				\lesssim& t^{-\frac{1}{2}}\int_{t_0}^t
				\tau^{\frac{1}{2}}\lp F_k(\tau,\cdot)\rp_{\dot{H}^{-1}(\R^n)}\md\tau \\
			\end{split}
		\end{equation}
		Collecting \eqref{equ:low-2-d} and \eqref{equ:high-2-d} we get
		\begin{equation}\label{equ:2-d}
			\begin{split}
				\|t^{\frac{1}{2}}  \nabla_x \phi_k \| _{L^2 (\mathbb R^n)}
				&\lesssim2^kt^{\frac{1}{2}}\lp\phi_k(t,\cdot)\rp_{L^2(\R^n)}\lesssim \int_{t_0}^t\tau^{\frac{1}{2}}\lp F_k(\tau,\cdot)\rp_{L^2(\R^n)}\md\tau. 
			\end{split}
		\end{equation}
		Summing up on both sides of \eqref{equ:2-d}, we get for \(\mu=1\)
		\begin{equation}\label{equ:2-d-1}
			\sup_{t_0<t\leq T}t^{\frac{1}{2}}\|\nabla_x\phi\|_{L^2(\mathbb R^n)}
			\lesssim \int_{t_0}^t\tau^{\frac{1}{2}}\lp F(\tau,\cdot)\rp_{L^2(\R^n)}\md\tau
			\lesssim \left\|\tau ^{\frac{1}{2}}\left(1+|u|  \right)^{\frac{\gamma}{2}}\left(1+\underline{u}  \right)^{\frac{\gamma}{2}}F\right\|_{L^2 ([t_0,T] \times \mathbb R^n)}.
		\end{equation}

		\noindent\textbf{Case ii \(\mathbf{\mu\in(0,1)\cup(1,2)}\)} In this case we have \(\nu=\frac{1-\mu}{2}\neq0\). If \(2^k\leq t^{-1}\), then
		\begin{equation*}
			\begin{split}
				\lp\phi_k(t,\cdot)\rp_{L^2(\R^n)}\lesssim&
				\int_{t_0}^{t}t^{\frac{1-\mu}{2}}\tau^{\frac{1+\mu}{2}}
				\lp(\tau|\xi|)^{-\frac{|1-\mu|}{2}}(t|\xi|)^{-\frac{|1-\mu|}{2}}
				\hat{F}_k(\tau,\xi)\rp_{L^{2}(\R^n)}\md\tau \\
				\lesssim&\left\{ \enspace
				\begin{aligned}
					&\int_{t_0}^{t}\tau^\mu2^{\mu k}\|F_k(\tau,\cdot)\|_{\dot{H}^{-1}(\R^n)}\ \md\tau\\
					&\lesssim t^{-\frac{\mu}{2}}\int_{t_0}^t
					\tau^{\frac{\mu}{2}}\lp F_k(\tau,\cdot)\rp_{\dot{H}^{-1}(\R^n)}\md\tau, \qquad \text{if}\ \ 0<\mu<1; \\
					&t^{-\frac{\mu}{2}}\int_{t_0}^t(2^kt)^{1-\frac{\mu}{2}}(2^k\tau)^{1-\frac{\mu}{2}}
					\tau^{\frac{\mu}{2}}\lp F_k(\tau,\cdot)\rp_{\dot{H}^{-1}(\R^n)}\md\tau\\
					&\lesssim t^{-\frac{\mu}{2}}\int_{t_0}^t
					\tau^{\frac{\mu}{2}}\lp F_k(\tau,\cdot)\rp_{\dot{H}^{-1}(\R^n)}\md\tau, \quad\text{if}\ \ 1<\mu<2.
				\end{aligned}
				\right. \\
			\end{split}
		\end{equation*}
		If \(2^k>t^{-1}\), then by making an extension of \(F\) such that \(F(\tau,y)\equiv0\) for \(\tau<t_0\)
		\begin{equation*}
			\begin{split}
				\lp\phi_k(t,\cdot)\rp_{L^2(\R^n)}\lesssim& \int_{0}^{2^{-k}}t^{\frac{1-\mu}{2}}\tau^{\frac{1+\mu}{2}}
				\lp(\tau|\xi|)^{-\frac{|1-\mu|}{2}}(t|\xi|)^{-\frac{1}{2}}
				|\hat{F}_k|\rp_{L^2(\R^n)}\md\tau  \\
				& +\int_{2^{-k}}^tt^{\frac{1-\mu}{2}}\tau^{\frac{1+\mu}{2}}
				\lp(\tau|\xi|)^{-\frac{1}{2}}(t|\xi|)^{-\frac{1}{2}}|\hat{F}_k|\rp_{L^2(\R^n)}\md\tau\\
				\lesssim& t^{-\frac{\mu}{2}}\int_{0}^{2^{-k}}\tau^{\frac{1+\mu-|1-\mu|}{2}}
				2^{-\frac{1+|\mu-1|}{2}k}\lp F_k(\tau,\cdot)\rp_{L^2(\R^n)}\md\tau\\
				&+t^{-\frac{\mu}{2}}
				\int_{2^{-k}}^t\tau^{\frac{\mu}{2}}
				\lp F_k(\tau,\cdot)\rp_{\dot{H}^{-1}(\R^n)}\md\tau \\
				\lesssim &t^{-\frac{\mu}{2}}\int_0^t
				\tau^{\frac{\mu}{2}}\lp F_k(\tau,\cdot)\rp_{\dot{H}^{-1}(\R^n)}\md\tau.
			\end{split}
		\end{equation*}
		By computation similar to \eqref{equ:2-d-1}, we have
		\begin{equation*}
			\sup_{{t_0}<t\leq T}t^{\frac{\mu}{2}}\|\nabla_x\phi\|_{L^2(\mathbb R^n)}
			\lesssim \left\|\tau ^{\frac{\mu}{2}}\left(1+|u|  \right)^{\frac{\gamma}{2}}\left(1+\underline{u}  \right)^{\frac{\gamma}{2}}F\right\|_{L^2 ([{t_0},T] \times \mathbb R^n)},\quad\text{for}\quad\mu\in(0,1)\cup(1,2).
		\end{equation*}
		
	\end{proof}
	
	\begin{lemma}\label{lem:energy-t}
		Consider the solution of \eqref{equ:linear-inhomo} with \(n\geq2\) and \(t_0\geq0\). Then for \(\gamma>1\) and \(\mu\in(0,2)\) we have the following energy estimates
		\begin{equation}\label{equ:energy-t}
			\underset{t_0< t\leq T}{\textup{sup}} t^{\frac{\mu}{2}} \left[  \left\|  \partial_t \phi \right\| _{L^2 (\mathbb R^n)}\right]\leq C
			\left \| \tau ^{\frac{\mu}{2}} \left(  1+ |u|  \right)^{\frac{\gamma}{2}} \left(  1+ \underline{u}  \right)^{\frac{\gamma}{2}} F \right\|  _{L^2 ([t_0,T] \times \mathbb R^n  )},
		\end{equation}
		where the positive constant \(C\) depends on \(n\), \(\mu\) and \(\gamma\).
	\end{lemma}
	\begin{proof}
		Since 
		\[\phi=\int_{t_0}^{t}\Phi_1(t,\tau,D)F\md\tau,\]
		and \(\Phi_1(t,t,D)=0\), we have
		\begin{equation*}
			\partial_t\phi=\int_{t_0}^{t}\partial_t\Phi_1(t,\tau,D)F\md\tau.
		\end{equation*}
		
		As in Lemma~\ref{lem:energy}, we estimate \(\|  \nabla_x \phi \| _{L^2 (\mathbb R^n)}\) by dyadic decomposition. Write
		\begin{equation*}
			\partial_t\phi_k=\int_{t_0}^{t}\partial_t\Phi_1(t,\tau,D)F_k\md\tau.
		\end{equation*}
		we will estimate \(\left\|  \partial_t \phi_k \right\| _{L^2 (\mathbb R^n)}\) by using different formulae of \(\partial_t\Phi_1(t,\tau,D)\) according to different ranges of \(\mu\) and \(2^k\).
		
		\noindent\textbf{Case i \(\mathbf{\mu=1}\)} If \(2^k\leq t^{-1}\), by \cite[(3.48)]{LiG25} we have 
		\begin{equation}\label{equ:t-Phi}
			\partial_t\Phi_1(t,\tau,\xi)=\tau|\xi|\ln\frac{t}{\tau}J_0(\tau|\xi|)J_{-1}(t|\xi|)
			+\frac{\pi}{2}\tau|\xi|\left[J_0(\tau|\xi|)A_{-1}(t|\xi|)-A_0(\tau|\xi|)J_{-1}(t|\xi|)\right],
		\end{equation}
		where \(J_\nu(z)\) is the Bessel function of the first kind with order \(\nu\) such that
		\begin{equation}\label{equ:Bessel-small}
			\begin{split}
				& J_\nu(z)\sim z^\nu, \quad \text{if} \quad z\rightarrow0,\ \nu\neq -1, -2, -3, \cdots \\
				& J_{-n}(z)=(-1)^nJ_n(z),  \quad \text{if} \quad n\in\mathbb{Z}.
			\end{split}
		\end{equation}
		While for \(n\in\mathbb{Z}\), \(z^nA_n(z)\) is entire function of \(z\) and non-zero for \(z=0\). With \eqref{equ:t-Phi} and \eqref{equ:Bessel-small}, we compute
		\begin{equation}\label{equ:a-11}
			\begin{split}
				\|\partial_t\phi_k\|_{L^2(\mathbb R^n)}&=\lp\int_{t_0}^{t}\partial_t\Phi_1(t,\tau,D)F_k\md\tau\rp_{L^2(\mathbb R^n)} \\
				&\lesssim\int_{t_0}^{t}\tau
				\lp|\xi|\left\{\ln\frac{t}{\tau}+1\right\}t|\xi|
				\hat{F}_k|\rp_{L^2(\mathbb R^n)}\md\tau,
			\end{split}
		\end{equation}
		Since \(2^k\leq t^{-1}\) (or \(t|\xi|\leq1\)), we have by \eqref{equ:a-11}
		\begin{equation}\label{equ:a-12}
			\begin{split}
				\|\partial_t\phi_k\|_{L^2(\mathbb R^n)}
				&\lesssim\int_{t_0}^{t}\tau
				\lp|\xi|\left\{\ln\frac{t}{\tau}+1\right\}(t|\xi|)^{-1}
				\hat{F}_k|\rp_{L^2(\mathbb R^n)}\md\tau \\
				&\lesssim\int_{t_0}^{t}\frac{\tau}{t}\left\{\ln\frac{t}{\tau}+1\right\}
				\lp F_k\rp_{L^2(\mathbb R^n)}\md\tau \\
				&\lesssim\int_{t_0}^{t}\left(\frac{\tau}{t}\right)^{\frac{1}{2}}
				\lp F_k\rp_{L^2(\mathbb R^n)}\md\tau.
			\end{split}
		\end{equation}
		
		If \(2^k>t^{-1}\), with out loss of generality we make an extension of \(F\) and assume that \(F(\tau,y)\equiv0\) if \(\tau<t_0\), then we write
		\begin{equation*}
			\partial_t\phi_k=\int_0^{2^{-k}}\partial_t\Phi_1(t,\tau,D)F_k\md\tau
			+\int_{2^{-k}}^t\partial_t\Phi_1(t,\tau,D)F_k\md\tau:=I+II.
		\end{equation*}
		By \cite[Corollary 2.2]{Wirth-0}, 
		\begin{equation}\label{equ:t-Phi-1}
			\partial_t\Phi_1(t,\tau,\xi)=-\frac{i\pi}{4}\tau|\xi|\left[H_0^-(\tau|\xi|)
			H_{-1}^+(t|\xi|)-H_0^+(\tau|\xi|)H_{-1}^-(t|\xi|)\right].
		\end{equation}
		For \(I\), we use the fact that
		\[H_\nu^\pm(z)=J_\nu(z)\pm iY_\nu(z),\]
		and
		\[Y_n(z)=\frac{2}{\pi}\ln zJ_n(z)+A_n(z).\]
		Thus we have
		\begin{equation*}
			\begin{split}
				I=&-\frac{i\pi}{4}\int_0^{2^{-k}}
				\tau|\xi|\bigg[\left(J_0(\tau|\xi|)-i\left(\frac{2}{\pi}\ln (\tau|\xi|)J_{0}(\tau|\xi|)+A_{0}(\tau|\xi|)\right)\right)H_{-1}^+(t|\xi|) \\
				&-\left(J_0(\tau|\xi|)+i\left(\frac{2}{\pi}\ln (\tau|\xi|)J_{0}(\tau|\xi|)+A_{0}(\tau|\xi|)\right)\right)H_{-1}^-(t|\xi|)\bigg],
			\end{split}   
		\end{equation*}
		by \eqref{equ:Bessel-small} and \eqref{equ:z-large-1} one can compute
		\begin{equation}\label{I}
			\|I\|_{L^2(\mathbb R^n)}\lesssim\int_{t_0}^{t}\left(\frac{\tau}{t}\right)^{\frac{1}{2}}
			\lp F_k\rp_{L^2(\mathbb R^n)}\md\tau.
		\end{equation}
		While \eqref{equ:z-large-1} implies that \(II\) can be estimated with the similar computation as in \eqref{equ:high-2-d}, then we get
		\begin{equation}\label{II}
			\|II\|_{L^2(\mathbb R^n)}\lesssim\int_{t_0}^{t}\left(\frac{\tau}{t}\right)^{\frac{1}{2}}
			\lp F_k\rp_{L^2(\mathbb R^n)}\md\tau.
		\end{equation}
		Collecting these results in \eqref{equ:a-11}, \eqref{I} and \eqref{II} and compute as in \eqref{equ:2-d-1}, we prove \eqref{equ:energy-t} for \(\mu=1\).
		
		\noindent\textbf{Case ii \(\mathbf{\mu\in(0,1)\cup(1,2)}\)} If \(2^k\leq t^{-1}\), we have by \cite[(23b)]{Wirth-0}
		\begin{equation*}
			\partial_t\Phi_1(t,\tau,\xi)=\frac{\pi}{2}\csc(\nu\pi)|\xi|\frac{t^\nu}{\tau^{\nu-1}}
			\left\{J_{-\nu}(\tau|\xi|)J_{\nu-1}(t|\xi|)-J_\nu(\tau|\xi|)J_{-\nu+1}(t|\xi|)\right\},
		\end{equation*}
		thus
		\begin{equation*}
			\begin{split}
				\|\partial_t\phi_k\|_{L^2(\mathbb R^n)}&=\lp\int_{t_0}^{t}\partial_t\Phi_1(t,\tau,D)F_k\md\tau\rp_{L^2(\mathbb R^n)} \\
				&\lesssim\int_{t_0}^{t}t^{\frac{1-\mu}{2}}\tau^{\frac{1+\mu}{2}}
				\lp|\xi|\left\{\left[(\tau|\xi|)^{-\frac{1-\mu}{2}}(t|\xi|)^{-\frac{1+\mu}{2}}
				-(\tau|\xi|)^{\frac{1-\mu}{2}}(t|\xi|)^{\frac{1+\mu}{2}}\right]|\hat{F}_k|\right\}\rp_{L^2(\mathbb R^n)}\md\tau \\
				&\lesssim\int_{t_0}^{t}\left\{t^{-\mu}\tau^\mu+t\tau2^{2k}\right\}\lp F_k(\tau,\cdot)\rp_{L^2(\R^n)}\md\tau \\
				&=\int_{t_0}^{t}\left\{t^{-\frac{\mu}{2}}\tau^{\frac{\mu}{2}}
				+(t2^k)^{1+\frac{\mu}{2}}(\tau2^k)^{1-\frac{\mu}{2}}\right\}
				t^{-\frac{\mu}{2}}\tau^{\frac{\mu}{2}}\lp F_k(\tau,\cdot)\rp_{L^2(\R^n)}\md\tau \\
				&\lesssim t^{-\frac{\mu}{2}}\int_{t_0}^{t}\tau^{\frac{\mu}{2}}\lp F_k(\tau,\cdot)\rp_{L^2(\R^n)}\md\tau.
			\end{split}
		\end{equation*}
		
		If \(2^k>t^{-1}\), the analysis is similar to that of \(\mu=1\), thus we omit the details and give
		\[\|\partial_t\phi_k\|_{L^2(\mathbb R^n)}\lesssim t^{-\frac{\mu}{2}}\int_{t_0}^{t}\tau^{\frac{\mu}{2}}\lp F_k(\tau,\cdot)\rp_{L^2(\R^n)}\md\tau, \quad\text{for}\quad 2^k>t^{-1}.\]
		Collecting these results above and compute as in \eqref{equ:2-d-1}, we prove \eqref{equ:energy-t} for \(\mu\in(0,1)\cup(1,2)\).
	\end{proof}

	\begin{lemma}\label{A2}
		Consider the solution of \eqref{equ:linear-inhomo} with \(n\geq2\), \(\mu=1\) and \(t_0=2\).
		Then for \(\delta>0\) which can be arbitrarily small,
		\begin{equation}\label{equ:cutoff-infty-1-1}
			\|\phi_k(t,\cdot)\|_{L^\infty(\R^n)}\leq C2^{\frac{n}{2}k}(\ln t)^2\int_{2}^{t}|t-\tau|^{-\frac{n}{2}}
			\tau\|F_k(\tau,\cdot)\|_{L^1(\R^n)}\ \md\tau.
		\end{equation}
		\begin{equation}\label{equ:cutoff-2-1-low}
			t\|\phi_k(t,\cdot)\|_{L^2(\R^n)}\leq C(\ln t)^2\int_{2}^{t}\tau\|F_k(\tau,\cdot)\|_{\dot{H}^{-1}(\R^n)}\ \md\tau, \quad \text{if} \quad 2^k\leq t^{-1},
		\end{equation}
		\begin{equation}\label{equ:cutoff-2-1-high}
			t^{\frac{1}{2}}\|\phi_k(t,\cdot)\|_{L^2(\R^n)}\leq C2^{\frac{1}{2}k}\ln t\int_{2}^{t}\tau
			\|F_k(\tau,\cdot)\|_{\dot{H}^{-1}(\R^n)}\ \md\tau, \quad \text{if} \quad 2^k> t^{-1}.
		\end{equation}
	\end{lemma}
	\begin{proof}

		\noindent\textbf{Case i: low frequency} Firstly we handle the case \(2^k\leq t^{-1}\). By \eqref{equ:phi-2}, we write
		\[\phi_k(t,x)=\int_{2}^{t}\int_{\R^n}\int_{\R^n}e^{i(x-y)\cdot\xi}
		\tau\big|\ln(\tau|\xi|)\ln(t|\xi|)\big|F_k(\tau,y)\md y\md\xi\md\tau,\]
		hence
		\begin{equation*}
			\begin{split}
				\lp \phi_k(t,\cdot)\rp_{L^{\infty}(\R^n)} &
				\lesssim\int_{2}^{t}k^22^{kn}\tau
				\lp F_k(\tau,\cdot)\rp_{L^1(\R^n)}\md\tau \\
				& \leq (\ln t)^2
				\int_{2}^{t}2^{\frac{n}{2}k}t^{-\frac{n}{2}}
				\tau\lp F_k(\tau,\cdot)\rp_{L^1(\R^n)}\md\tau \\
				& \leq2^{\frac{n}{2}k}(\ln t)^2\int_{2}^{t}|t-\tau|^{-\frac{n}{2}}
				\tau\lp F_k(\tau,\cdot)\rp_{L^1(\R^n)}\md\tau.
			\end{split}
		\end{equation*}
		Similarly we have
		\begin{equation*}
			\begin{split}
				\lp \phi_k(t,\cdot)\rp_{L^{2}(\R^n)} & \leq
				\int_{2}^{t}\tau\lp\big|\ln(\tau|\xi|)\ln(t|\xi|)\big|
				\hat{F}_k(\tau,\xi)\rp_{L^{2}(\R^n)}\md\tau \\
				& \leq (\ln t)^2t^{-1}\int_{2}^{t}\tau\lp
				F_k(\tau,\cdot)\rp_{\dot{H}^{-1}(\R^n)}\md\tau.
			\end{split}
		\end{equation*}

		\noindent\textbf{Case ii: high frequency} Next we consider \(2^k\geq t^{-1}\), with out loss of generality we make an extension of \(F\) and assume that \(F(\tau,y)\equiv0\) if \(\tau<2\). Then by \eqref{equ:z-large-1} with \(K=1\) we have
		\[
		\begin{split}
			\phi_k(t,x) & =\int_{2}^{t}\int_{\R^n}\int_{\R^n}
			e^{i(x-y)\cdot\xi}
			\tau H_0^-(\tau|\xi|)H_0^+(t|\xi|)F_k(\tau,y)\md y\md\xi\md\tau, \\
			& =\int_{0}^{2^{-k}}\int_{\R^n}\int_{\R^n}
			e^{i(x-y)\cdot\xi+it|\xi|}
			\tau \ln(\tau|\xi|)a_0^+(t|\xi|)F_k(\tau,y)\md y\md\xi\md\tau \\
			&+\int_{2^{-k}}^t\int_{\R^n}\int_{\R^n}
			e^{i(x-y)\cdot\xi+i(t-\tau)|\xi|}
			\tau a_0^-(\tau|\xi|)a_0^+(t|\xi|)F_k(\tau,y)\md y\md\xi\md\tau,
		\end{split}
		\]
		where \(a_0^-, a_0^+\in S^{-\frac{1}{2}}([1,\infty))\). Then by stationary phase method we have
		\begin{equation*}
			\begin{split}
				\lp\phi_k(t,\cdot)\rp_{L^{\infty}(\R^n)} & \lesssim 2^{nk}\int_{0}^{2^{-k}}\tau\ln(\tau|\xi|)
				(t|\xi|)^{-\frac{1}{2}}
				(1+t|\xi|)^{-\frac{n-1}{2}}|\hat{F}_k|\md\tau\\
				& +2^{nk}\int_{2^{-k}}^t\tau(\tau|\xi|)^{-\frac{1}{2}}
				(t|\xi|)^{-\frac{1}{2}}
				(1+|t-\tau||\xi|)^{-\frac{n-1}{2}}|\hat{F}_k|\md\tau \\
				& \lesssim \ln t\,2^{\frac{n}{2}k}
				\int_{0}^{2^{-k}}|t-\tau|^{-\frac{n}{2}}
				\lp F_k(\tau,\cdot)\rp_{L^1(\R^n)}\md\tau \\
				& +2^{\frac{n}{2}k}
				\int_{2^{-k}}^t|t-\tau|^{-\frac{n}{2}}\tau
				\lp F_k(\tau,\cdot)\rp_{L^1(\R^n)}\md\tau \\
				& \leq2^{\frac{n}{2}k}\ln t\int_0^t|t-\tau|^{-\frac{n}{2}}
				\tau\lp F_k(\tau,\cdot)\rp_{L^1(\R^n)}\md\tau \\
				& =2^{\frac{n}{2}k}\ln t\int_2^t|t-\tau|^{-\frac{n}{2}}
				\tau\lp F_k(\tau,\cdot)\rp_{L^1(\R^n)}\md\tau.
			\end{split}
		\end{equation*}
		While direct computation gives the \(L^2\) estimate
		\begin{equation*}
			\begin{split}
				& \lp\phi_k(t,\cdot)\rp_{L^2(\R^n)} \\
				\lesssim &\int_{0}^{2^{-k}}\tau\lp\ln(\tau|\xi|)
				(t|\xi|)^{-\frac{1}{2}}|\hat{F}_k|\rp_{L^2(\R^n)}\md\tau +\int_{2^{-k}}^t\tau\lp(\tau|\xi|)^{-\frac{1}{2}}
				(t|\xi|)^{-\frac{1}{2}}|\hat{F}_k|\rp_{L^2(\R^n)}\md\tau \\
				\lesssim &t^{-\frac{1}{2}}\ln t2^{\frac{1}{2}k}
				\int_{0}^{2^{-k}}\tau
				\lp F_k(\tau,\cdot)\rp_{\dot{H}^{-1}(\R^n)}\md\tau
				+t^{-\frac{1}{2}}2^{\frac{1}{2}k}
				\int_{2^{-k}}^t
				\tau\lp F_k(\tau,\cdot)\rp_{\dot{H}^{-1}(\R^n)}\md\tau \\
				\lesssim &2^{\frac{1}{2}k}t^{-\frac{1}{2}}\ln t
				\int_0^t
				\tau\lp F_k(\tau,\cdot)\rp_{\dot{H}^{-1}(\R^n)}\md\tau.
			\end{split}
		\end{equation*}
	\end{proof}

	Next we introduce a result of dyadic decomposition from Lemma 3.8 in \cite{Gls1}.
	
	\begin{lemma}\label{lemma:a1}
		For the cut-off function \(\rho\) in \eqref{equ:cut-off},
		define the Littlewood-Paley operators in function $G$ as follows
		\[G_j(t,x)=(2\pi)^{-n}\int_{\mathbb{R}^n}e^{ix\cdot\xi}
		\rho_j(\xi)\hat{G}(t,\xi)d\xi.\]
		Then one has that
		\begin{align*}
			\parallel G\parallel_{L^s_tL^q_x}\leq C\lc \sum\limits_{j=-\infty}^{\infty}\parallel G_j\parallel^2_{L^s_tL^q_x}\rc^{\frac{1}{2}}
			\qquad \text{for\quad $2\leq q<\infty$ and $2\leq s \leq \infty$} \\
		\end{align*}
		and
		\begin{align*}
			\lc\sum\limits_{j=-\infty}^{\infty}\parallel G_j\parallel^2_{L^r_tL^p_x}\rc^{\frac{1}{2}}\leq C\parallel G\parallel_{L^r_tL^p_x}
			\qquad \text{for\quad $1<p\leq2$\quad and \quad $1\leq r \leq 2$}.
		\end{align*}
	\end{lemma}
	
	Finally, with Lemma~\ref{lemma:a1} in hand we can prove the inhomogeneous Strichartz estimate without characteristic weight at the endpoint \(q=q_0=\frac{2(n+2)}{n}\).
	
	\begin{lemma}\label{lem:time-weighted-Strichartz}
		Suppose that \(n\geq2\) and \(\phi\) solves the Cauchy problem \eqref{equ:linear-inhomo} with \(\mu=1\) and \(t_0=2\), then for \(\delta>0\) which can be arbitrarily small, it holds
		\begin{equation*}
			\lp \frac{t^{\frac{1}{q_0}}}{\big(\ln (2+t)\big)^2} \phi\rp_{L^{q_0}(\ra)}\leq C(\delta_1,\delta_2)\lp \tau^{\frac{1}{p_0}}F\rp
			_{L^{p_0}(\ra)},
		\end{equation*}
		where \(q_0=\frac{2(n+2)}{n}\) and \(\frac{1}{p_0}+\frac{1}{q_0}=1\).
	\end{lemma}
	\begin{proof}
		Using the idea from \cite{Gls2}, we define
		\[H^j_{t,\tau}F(x)=\int_{\R^n}\int_{\R^n}e^{i(x-y)\cdot\xi}
		\rho_j(\xi)\tau H_0^-(\tau|\xi|)H_0^+(t|\xi|)F(\tau,y)\md y\md\xi,\]
		and
		\[A_jF(t,x)=\int_{2}^{t}H^j_{t,\tau}F(x)\md\tau,\quad
		F_k(t,x)=\int_{\R^n}\int_{\R^n}e^{i(x-y)\cdot\xi}
		\rho(2^{-k}|\xi|)F(\tau,y)\md y\md\xi,\]
		where \(\rho\) is the cut off function defined in \eqref{equ:cut-off}. Then \[\phi(t,x)=\sum_{j=-\infty}^{\infty}A_jF(t,x),
		\quad F(\tau,y)=\sum_{k=-\infty}^{\infty}F_k(\tau,y).\]
		Next we estimate \(H^j_{t,\tau}F\) with respect to the scale of \(\lambda_j=2^j\).
		
		\noindent\textbf{Case i: high frequency} We first consider \(\lambda_j\geq\frac{1}{\tau}\geq\frac{1}{t}\), which implies
		\begin{equation*}
			H^j_{t,\tau}F(x)=\int_{\R^n}\int_{\R^n}e^{i(x-y)\cdot\xi+i(t-\tau)\xi}
			\rho_j(\xi)\tau(\tau|\xi|)^{-\frac{1}{2}}(t|\xi|)^{-\frac{1}{2}}
			F(\tau,y)\md y\md\xi,
		\end{equation*}
		then by the stationary phase method we get
		\begin{equation}\label{equ:L-infty-high-1}
			\begin{split}
				\| H^j_{t,\tau}F(x)\|_{L^\infty(\R^n)} & \lesssim
				2^{-\frac{1}{2}j}t^{-\frac{1}{2}}2^{nj}(1+2^j|t-\tau|)^{-\frac{n-1}{2}}
				\|\tau F(\tau,\cdot)\|_{L^1(\R^n)} \\
				&  \lesssim2^{\frac{n}{2}j}|t-\tau|^{-\frac{n}{2}}
				\|\tau F(\tau,\cdot)\|_{L^1(\R^n)}.
			\end{split}
		\end{equation}
		While direct computation yields
		\begin{equation}\label{equ:L-2-high-1}
			\begin{split}
				\|t^{\frac{1}{2}}H^j_{t,\tau}F(x)\|_{L^2(\R^n)} & \lesssim
				2^{-j}\|\tau^{\frac{1}{2}} F(\tau,\cdot)\|_{L^2(\R^n)}.
			\end{split}
		\end{equation}
		Interpolating between \eqref{equ:L-infty-high-1} and \eqref{equ:L-2-high-1}, we obtain
		\begin{equation*}
			\lp t^{\frac{1}{q_0}}H^j_{t,\tau}F(\cdot)\rp_{L^{q_0}(\R^n)}
			\leq|t-\tau|^{-\frac{n}{n+2}}\lp\tau^{\frac{1}{p_0}}
			F(\tau,\cdot)\rp_{L^{p_0}(\R^n)}, \quad\text{if}\quad 2^j\tau\geq1.
		\end{equation*}
		
		\noindent\textbf{Case ii: low frequency} Next we handle \(\lambda_j\leq\frac{1}{t}\leq\frac{1}{\tau}\). In this case,
		\begin{equation*}
			H^j_{t,\tau}F(x)=\int_{\R^n}\int_{\R^n}e^{i(x-y)\cdot\xi}
			\rho_j(\xi)\tau\ln(\tau|\xi|)\ln(t|\xi|)
			F(\tau,y)\md y\md\xi,
		\end{equation*}
		then direct computation yields
		\begin{equation}\label{equ:L-infty-low-1}
			\begin{split}
				\| H^j_{t,\tau}F(x)\|_{L^\infty(\R^n)} & \lesssim
				|j|^22^{nj}
				\left\|\tau F(\tau,\cdot)\right\|_{L^1(\R^n)} \\
				&  \lesssim t^{-\delta}\big(\ln (2+t)\big)^22^{(n-\delta)j}(2^jt)^{-\frac{n}{2}+\delta}
				\left\|\tau F(\tau,\cdot)\right\|_{L^1(\R^n)} \\
				&  \lesssim2^{\frac{n}{2}j}\big(\ln (2+t)\big)^2|t-\tau|^{-\frac{n}{2}}
				\left\|\tau F(\tau,\cdot)\right\|_{L^1(\R^n)}.
			\end{split}
		\end{equation}
		and
		\begin{equation}\label{equ:L-2-low-1}
			\begin{split}
				\left\|t^{\frac{1}{2}}H^j_{t,\tau}F(x)\right\|_{L^2(\R^n)} &
				\lesssim t^{\frac{1}{2}}|j|^22^j\left\|\tau F(\tau,\cdot)\right\|_{\dot{H}^{-1}(\R^n)} \\
				&\lesssim t^{\frac{1}{2}}t^{-\frac{1}{2}}\big(\ln (2+t)\big)^2
				\tau^{-\frac{1}{2}}\left\|\tau F(\tau,\cdot)\right\|_{\dot{H}^{-1}(\R^n)} \\
				& \lesssim
				2^{-j}\big(\ln (2+t)\big)^2\left\|\tau^{\frac{1}{2}} F(\tau,\cdot)\right\|_{L^2(\R^n)}.
			\end{split}
		\end{equation}
		Interpolating between \eqref{equ:L-infty-low-1} and \eqref{equ:L-2-low-1}, we obtain
		\begin{equation*}
			\begin{split}
				\lp t^{\frac{1}{q_0}-\delta}H^j_{t,\tau}F(\cdot)\rp_{L^{q_0}(\R^n)}
				\leq&\big(\ln (2+t)\big)^2
				|t-\tau|^{-\frac{n}{n+2}}\lp\tau^{\frac{1}{p_0}}
				F(\tau,\cdot)\rp_{L^{p_0}(\R^n)}, \quad\text{if}\quad 2^jt\leq1.
			\end{split}
		\end{equation*}
		
		\noindent\textbf{Case iii: medium frequency} The remaining case is
		\(\frac{1}{t}\leq\lambda_j\leq\frac{1}{\tau}\), which can be treated by the method in Case i and Case ii, thus we omit the details.
		
		Collecting the results together, we have
		\[\lp t^{\frac{1}{q_0}}\big(\ln (2+t)\big)^{-2}H^j_{t,\tau}F(\cdot)\rp_{L^{q_0}(\R^n)}
		\leq|t-\tau|^{-\frac{n}{n+2}}\lp\tau^{\frac{1}{p_0}}
		F(\tau,\cdot)\rp_{L^{p_0}(\R^n)}, \quad\text{for}\quad j\in\mathbb{Z}.\]
		Then by setting $1-(\frac{1}{p_0}-\frac{1}{q_0})=
		\frac{n}{n+2}$ and the
		Hardy-Littlewood-Sobolev inequality, we arrive at
		\begin{equation*}
			\begin{aligned}
				\lp t^{\frac{1}{q_0}}\big(\ln (2+t)\big)^{-2}A_jF\rp_{L^{q_0}(\ra)}=&\left\|\int_{2}^{\infty}t^{\frac{1}{q_0}}\big(\ln (2+t)\big)^{-2}H_{t,\tau}^jF\,d\tau \right\|_{L^{q_0}(\ra)}\\
				\leq &C\left\|\int_{2}^{\infty}
				|t-\tau|^{-\frac{n}{n+2}}\left\|\tau^{\frac{1}{p_0}} F(\tau,\cdot)\right\|_{L^{p_0}(\R^n)}\,
				\md\tau\right\|_{L^{q_0}([T_0,\infty))} \\
				\leq &C\left\|t^{\frac{1}{p_0}}F\right\|_{L^{p_0}(\ra)}.
			\end{aligned}
		\end{equation*}
		
		Then an application of Lemma~\ref{lemma:a1} implies
		\begin{equation*}
			\begin{aligned}
				\|t^{\frac{1}{q_0}}\big(\ln (2+t)\big)^{-2}\phi\|_{L^{q_0}}^2\leq &C\sum_{j\in\Z}\| t^{\frac{1}{q_0}}\big(\ln (2+t)\big)^{-2}H^j_{t,\tau}F\|_{L^{q_0}}^2 \\
				\leq &C\sum_{j\in\Z}\sum_{k:|j-k|\leq C_0}\|t^{\frac{1}{q_0}}\big(\ln (2+t)\big)^{-2}H^j_{t,\tau}F_k\|_{L^{q_0}}^2 \\
				\leq &C\sum_{j\in\Z}\sum_{k:|j-k|\leq C_0}\|t^{\frac{1}{p_0}}F_k\|_{L^{p_0}}^2\\
				\le& C\, \|t^{\frac{1}{p_0}}F\|_{L^{p_0}(\ra)}^2.
			\end{aligned}
		\end{equation*}
	\end{proof}

\end{document}